\documentclass[10pt]{amsart}

\usepackage[margin=1in]{geometry}
\usepackage{amsmath}
\usepackage{float}
\usepackage{booktabs}
\usepackage[dvipsnames]{xcolor} 

\usepackage{fix-cm}

\usepackage{wasysym}

\usepackage{blkarray}

\usepackage{amscd,amsmath,amssymb,amsfonts,amsthm}
\usepackage{enumerate,mathrsfs,stmaryrd,latexsym, comment, mathtools, mathdots} 

\usepackage[mathscr]{euscript}

\usepackage{caption}
\usepackage{subcaption}
\usepackage{hyperref}

\usepackage[all]{xy}

\usepackage{tikz-cd}
\usepackage{tikz}
\usetikzlibrary{snakes, 
	3d, matrix,decorations.pathreplacing,calc,arrows,decorations.pathmorphing, patterns}
\usetikzlibrary {positioning}
\usetikzlibrary{calc,backgrounds}

\pgfdeclarelayer{background}
\pgfdeclarelayer{foreground}
\pgfsetlayers{background,main,foreground}

\usepackage{tikz-3dplot}

\numberwithin{equation}{section}
\allowdisplaybreaks[1]

\makeatletter
\newcommand{\leqnomode}{\tagsleft@true\let\veqno\@@leqno}
\newcommand{\reqnomode}{\tagsleft@false\let\veqno\@@eqno}
\makeatother

\DeclareMathOperator{\SL}{SL}
\DeclareMathOperator{\GL}{GL}
\DeclareMathOperator{\Gr}{Gr}

\DeclareMathOperator{\Hom}{Hom}
\DeclareMathOperator{\Sym}{Sym}

\DeclareMathOperator{\SO}{SO}

\DeclareMathOperator{\diag}{diag}
\DeclareMathOperator{\Ric}{Ric}

\newtheorem{theorem}{Theorem}[section]
\newtheorem{lemma}[theorem]{Lemma}
\newtheorem{proposition}[theorem]{Proposition}
\newtheorem{corollary}[theorem]{Corollary}

\newtheorem{Question}[theorem]{Question}

\theoremstyle{definition}
\newtheorem{example}[theorem]{Example}
\newtheorem{definition}[theorem]{Definition}
\newtheorem{remark}[theorem]{Remark}

\begin{document}

\title[K\"ahler--Einstein metrics on  smooth Fano toroidal symmetric varieties of type AIII]{K\"ahler--Einstein metrics on \\ smooth Fano toroidal symmetric varieties of type AIII}	

\author{Kyusik Hong}
\address{Department of Mathematics Education, Jeonju University, Jeonju 55069, Republic of Korea}
\email{kszooj@jj.ac.kr}

\author{DongSeon Hwang}
\address{Center for Complex Geometry, Institute for Basic Science (IBS), Daejeon 34126, Republic of Korea}
\email{dshwang@ibs.re.kr}

\author{Kyeong-Dong Park}
\address{Department of Mathematics and Research Institute of Natural Science/Research Institute of Molecular Alchemy, Gyeongsang National University, Jinju 52828, Republic of Korea}
\email{kdpark@gnu.ac.kr}



\begin{abstract}
The wonderful compactification $X_m$ of a symmetric homogeneous space of type AIII$(2,m)$ for each $m \geq 4$ is Fano, and its blowup $Y_m$ along the unique closed orbit is Fano if $m \geq 5$ and Calabi--Yau if $m = 4$. 
Using a combinatorial criterion for K-polystability of smooth Fano spherical varieties obtained by Delcroix, we prove that $X_m$ admits a K\"ahler--Einstein metric for each $m \geq 4$ and $Y_m$ admits a K\"ahler--Einstein metric if and only if $m = 4,5$. 
\end{abstract}

\maketitle
\setcounter{tocdepth}{1} 
\date{\today}


\section{Introduction}

A K\"{a}hler metric on a complex manifold is said to be \emph{K\"{a}hler--Einstein} if it satisfies the Einstein condition $$\Ric (\omega) = \lambda \, \omega,$$ 
where $\lambda \in \mathbb R$ and $\Ric (\omega)$ denotes the Ricci curvature form of the associated K\"{a}hler form $\omega$.  
While Calabi--Yau manifolds and K\"{a}hler manifolds with ample canonical bundle always admit a K\"{a}hler--Einstein metric by the works of Aubin~\cite{aubin78} and Yau~\cite{yau78}, Fano manifolds do not necessarily admit a K\"{a}hler--Einstein metric in general. 
By the Yau--Tian--Donaldson conjecture solved by Chen--Donaldson--Sun~\cite{cds1, cds2, cds3} and Tian~\cite{tian15}, 
the existence of a K\"{a}hler--Einstein metric on a Fano manifold has turned out to be equivalent to the \emph{K-polystability} which is an algebro-geometric condition introduced by Tian~\cite{tian94, tian97} and Donaldson~\cite{donaldson02}. 

However, it is a difficult problem to determine K-polystability of a given Fano manifold. 
Most known results are concentrated on del Pezzo surfaces and Fano $3$-folds by focusing on each individual manifold or a general member of a finitely many families of Fano manifolds (\cite{TY87}, \cite{TY90}, and \cite{ACCFKMGSSV}) except for the recent breakthrough \cite{AZ} and \cite{Zhuang21}.

Fortunately, more manageable K-polystability criterion have been found for Fano varieties equipped with algebraic group actions. 
For instance, the works of Mabuchi~\cite{mabuchi} together with that of Wang and Zhu~\cite{wz} imply that a toric Fano manifold is K-polystable if and only if the barycenter of the moment polytope is the origin. 
This gave us a powerful combinatorial criterion for the K-polystability of a toric Fano manifold.

This is largely generalized to smooth Fano spherical varieties in terms of its moment polytope and spherical data by Delcroix (\cite{Del20}). 
More precisely, a smooth Fano spherical variety $X$ is K-polystable if and only if the barycenter of the moment polytope of $X$ with respect to the Duistermaat–Heckman measure belongs to certain cone. 
See also \cite{Del22} for some examples including the blow-up of a quadric along a linear subquadric. 
Here, a normal variety is called \emph{spherical} if it admits an action of a reductive algebraic group whose Borel subgroup acts with an open orbit on the variety.

If a Fano spherical manifold is homogeneous, it is K-polystable by \cite[Section 5]{mat2}. 
There are two additional well-known classes of spherical varieties: symmetric varieties and horospherical varieties. 
It is known that, under the assumption that it is a non-homogeneous Fano manifold of Picard number one, the former is K-polystable by \cite{LPY21} and the latter is not since the automorphism group is non-reductive by \cite[Theorem 0.1]{Pasquier}. 
Thus one might be interested in determining K-polystability of non-homogeneous Fano symmetric manifolds of Picard number two or higher. 

Recall that \emph{symmetric varieties} are normal equivariant open embeddings of (algebraic) symmetric homogeneous spaces (see Section \ref{criterion for symmetric varieties} for details). 
From \cite[Theorem~3.1]{dCP83}, we know that a symmetric homogeneous space $G/N_G(G^{\theta})$ admits a maximal smooth equivariant compactification having a unique closed $G$-orbit, which is called the \emph{wonderful compactification} of $G/N_G(G^{\theta})$. 
A symmetric variety without colors is said to be \emph{toroidal}, and a toroidal symmetric variety of $G/N_G(G^{\theta})$ dominates the wonderful compactification of $G/N_G(G^{\theta})$ by \cite[Section~2.3]{dCP83}. 
If a Fano symmetric variety is obtainable from a wonderful compactification $S$ by a sequence of blow-ups along closed orbits, then it is either $S$ or the blow-up of $S$ along the unique closed orbit by \cite[Theorem B]{Ruzzi2012}. 
In summary, one can ask the following question:

\begin{Question}
Is the wonderful compactification of a symmetric homogeneous space K-polystable?  
Which of the blow-ups of wonderful compactifications of symmetric homogeneous spaces along the (unique) closed orbit are K-polystable? 
\end{Question}

In this paper, we answer the question in the case of the wonderful compactifications $X_m$ of symmetric homogeneous spaces of type AIII$(2, m)$ for $m \geq 4$ and their blow-ups $Y_m$ along the closed orbit. 
It is known that $X_m$ is Fano for each $m \geq 4$ and $Y_m$ is Fano if $m \geq 5$ and Calabi--Yau if $m = 4$. 
Recall that $X_m$ is of Picard number two and $Y_m$ is of Picard number three. 
Geometrically, for $m \geq 5$, the wonderful compactification of symmetric homogeneous space 
$$\SL_m(\mathbb C)/S(\GL_2(\mathbb C) \times \GL_{m-2}(\mathbb C))$$ 
of type $\textup{AIII}(2, m)$ can be described as the blow-up of the product $\Gr(2, m) \times \Gr(m-2, m)$ of Grassmannians along the closed orbit $\mathcal{F \ell}(2, m-2; m)$. 
In the case of $m=4$, the symmetric homogeneous space $\SL_4(\mathbb C)/N_{\SL_4(\mathbb C)}(S(\GL_2(\mathbb C) \times \GL_2(\mathbb C)))$ admits the wonderful compactification. 

\begin{theorem}
\label{Main theorem} 
All wonderful compactifications $X_m$ of symmetric homogeneous spaces of type $\textup{AIII}(2, m)$ for $m \geq 4$ are K-polystable.
\end{theorem}

For the intersection $Z$ of the two $G$-stable divisors $Y_1$ and $Y_2$ in $X_m$, it turns out that the blow-up $B\ell_Z(X_m)$ of $X_m$ along the closed orbit $Z$ admits a K\"ahler--Einstein metric if and only if $m=4, 5$. 
Note that when $m=4$, $B\ell_Z(X_m)$ is Calabi--Yau, so it admits a K\"ahler--Einstein metric. 

\begin{theorem}
\label{Blowup}
Let $X_m$ be the wonderful compactification of symmetric homogeneous space of type $\textup{AIII}(2, m)$ for $m \geq 5$. 
The blow-up $B\ell_Z(X_m)$ of $X_m$ along the closed orbit $Z$ is K-polystable if $m=5$ and K-unstable if $m \geq 6$.  
\end{theorem}

In fact, we show that, in Theorems \ref{Main theorem-eq} and \ref{Blowup-eq}, both $X_m$ for $m \geq 4$ and $B\ell_Z(X_m)$ of $X_m$ for $m=5$ are  $\SL_m(\mathbb C)$-equivariantly uniformly K-stable, not just K-polystable, which implies Theorems \ref{Main theorem} and \ref{Blowup} by Proposition~\ref{criterion}. 
Here, we emphasize that Theorems \ref{Main theorem} and \ref{Blowup} are one of the few cases, including \cite{Del22} and \cite{LL}, that the K-stability or K-unstability of a reasonable class consisting of infinitely many Fano manifolds of different dimensions is verified. 
In practice, the estimate of the barycenter becomes extremely hard when considering \emph{infinite} series of the moment polytope of Fano spherical manifolds.

The paper is organized as follows.
In Section~2, we recall the theory of spherical varieties by focusing on symmetric varieties, and we state a combinatorial criterion for $G$-equivariant K-stability of smooth Fano spherical varieties obtained by Delcroix in terms of algebraic moment polytopes. 
In Section~3, we study the wonderful compactification $X_m$ of symmetric homogeneous space of type $\textup{AIII}(2, m)$ for $m \geq 4$, 
and prove Theorem~\ref{Main theorem} by computing the barycenter of the moment polytope of $X_m$ with respect to the Duistermaat--Heckman measure.  
In Section~4, we prove Theorem~\ref{Blowup} and compute the greatest Ricci lower bound of the blow-up of $X_6$ along the closed orbit.

\vskip 1em

\noindent
\textbf{Acknowledgments}. 
The authors would like to thank JongHae Keum for his interest and encouragement. 
They would like to express their gratitude to the anonymous referees for carefully reading the manuscript and providing excellent suggestions for improvement. 

Kyusik Hong was supported by the National Research Foundation of Korea(NRF) grant funded by the Korea government(MSIT) (NRF-2019R1A2C3010487, NRF-2021R1F1A1059506). 
DongSeon Hwang was supported by the Samsung Science and Technology Foundation under Project SSTF-BA1602-03, the National Research Foundation of Korea(NRF) grant funded by the Korea government(MSIT) (NRF-2021R1A2C1093787), and the Institute for Basic Science(IBS-R032-D1). 
Kyeong-Dong Park was supported by the National Research Foundation of Korea(NRF) grant funded by the Korea government(MSIT) (NRF-2019R1A2C3010487, NRF-2021R1C1C2092610), and by the fund of research promotion program, Gyeongsang National University, 2022. 
He was also supported by Learning \& Academic research institution for Master’s·PhD students, and Postdocs(LAMP) Program of the National Research Foundation of Korea(NRF) grant funded by the Ministry of Education(RS-2023-00301974). 


\section{Spherical varieties and symmetric varieties} 

Let $G$ be a connected reductive algebraic group over $\mathbb C$. 

\subsection{Spherical varieties and algebraic moment polytopes} 
We review general notions and results about spherical varieties. 
We refer \cite{Knop91}, \cite{Timashev11} and \cite{Gandini18} as references for spherical varieties. 

\begin{definition}
\label{spherical variety}
A normal variety $X$ equipped with an action of $G$ is called \emph{spherical} if a Borel subgroup $B$ of $G$ acts on $X$ with an open orbit.
\end{definition}

Let $G/H$ be an open dense $G$-orbit of a spherical variety $X$ and $T$ a maximal torus of $B$. 
By definition, the \emph{spherical weight lattice} $\mathcal M$ of $G/H$ is 
a subgroup of characters $\chi \in \mathfrak X(B) = \mathfrak X(T)$ of (nonzero) $B$-semi-invariant functions in the function field $\mathbb C(G/H) = \mathbb C(X)$, 
that is, $$\mathcal M = \{ \chi \in \mathfrak X(T) : \mathbb C(G/H)^{(B)}_{\chi} \neq 0 \},$$ 
where $\mathbb C(G/H)^{(B)}_{\chi} = \{ f \in \mathbb C(G/H) : b \cdot f = \chi(b) f \text{ for all } b \in B \}$. 
Note that every function $f_{\chi}$ in $\mathbb C(G/H)^{(B)}$ is determined by its weight $\chi$ up to constant 
because $\mathbb C(G/H)^{B} = \mathbb C$, that is, any $B$-invariant rational function on $X$ is constant. 
The spherical weight lattice $\mathcal M$ is a free abelian group of finite rank. 
We define the \emph{rank} of $G/H$ as the rank of the lattice $\mathcal M$. 
Let $\mathcal N = \Hom(\mathcal M, \mathbb Z)$ denote its dual lattice together with the natural pairing $\langle \, \cdot \, , \, \cdot \, \rangle \colon \mathcal N \times \mathcal M \to \mathbb Z$. 

Let $L$ be a $G$-linearized ample line bundle on a spherical $G$-variety $X$. 
By the multiplicity-free property of spherical varieties, 
the algebraic moment polytope $\Delta(X, L)$ encodes the structure of representation of $G$ in the spaces of multi-sections of tensor powers of $L$. 

\begin{definition} 
The \emph{algebraic moment polytope} $\Delta(X, L)$ of $L$ with respect to $B$ is defined as the closure of $\displaystyle \bigcup_{k \in \mathbb N} \Delta_k / k$ in $\mathfrak X(T) \otimes \mathbb R$, 
where $\Delta_k$ is a finite set consisting of (dominant) weights $\lambda$ such that 
\begin{equation*}
H^0(X, L^{\otimes k}) = \bigoplus_{\lambda \in \Delta_k} V_G(\lambda).
\end{equation*} 
Here, $V_G(\lambda)$ means the irreducible representation of $G$ with highest weight $\lambda$. 
\end{definition}

The algebraic moment polytope $\Delta(X, L)$ for a polarized (spherical) $G$-variety $X$ was introduced by Brion in \cite{Brion87} as a purely algebraic version of the Kirwan polytope. 
This is indeed the convex hull of finitely many points in $\mathfrak X(T) \otimes \mathbb R$ (see \cite{Brion87}).

\subsection{Colors of spherical varieties and types of colors}
As the open $B$-orbit of a spherical variety $X$ is an affine variety, 
its complement has pure codimension one and is a finite union of $B$-stable prime divisors. 

\begin{definition}
\label{color}
For a spherical homogeneous space $G/H$, 
$B$-stable prime divisors in $G/H$ are called \emph{colors} of $G/H$. 
We denote by $\mathfrak D = \{ D_1, \cdots, D_k \}$ the set of colors of $G/H$.
\end{definition}

As a $B$-semi-invariant function $f_{\chi}$ in $\mathbb C(G/H)^{(B)}_{\chi}$ is unique up to constant, 
we define the \emph{color map} $\rho \colon \mathfrak D \to \mathcal N$ by $\langle \rho(D), \chi \rangle = \nu_D(f_{\chi})$ for $\chi \in \mathcal M$, 
where $\nu_D$ is the discrete valuation associated to a divisor $D$, that is, $\nu_D(f)$ is the vanishing order of $f$ along $D$. 
Unfortunately, the color map is generally not injective.
In addition, every discrete $\mathbb Q$-valued valuation $\nu$ of the function field $\mathbb C(G/H)$ induces a homomorphism $\hat{\rho}(\nu) \colon \mathcal M \to \mathbb Q$ defined by $\langle \hat{\rho}(\nu), \chi \rangle = \nu(f_{\chi})$, so that we get a map $\hat{\rho} \colon \{ \text{discrete $\mathbb Q$-valued valuations on $G/H$} \} \to \mathcal N \otimes \mathbb Q$. 
Luna and Vust \cite{LV83} showed that the restriction of $\hat{\rho}$ to the set of $G$-invariant discrete valuations on $G/H$ is injective. 
From now on, we will regard a $G$-invariant discrete valuation on $G/H$ as an element of $\mathcal N \otimes \mathbb Q$ via the map $\hat{\rho}$, 
and in order to simplify the notation $\hat{\rho}(\nu_E)$ will be written as $\hat{\rho}(E)$ for a $G$-stable divisor $E$ in $X$. 

Let $\mathcal V$ be the set of $G$-invariant discrete $\mathbb Q$-valued valuations on $G/H$. 
Since the map $\hat{\rho}$ is injective, we may consider $\mathcal V$ as a subset of $\mathcal N \otimes \mathbb Q$ 
and it is known that $\mathcal V$ is a full-dimensional cosimplicial cone, called the \emph{valuation cone} of $G/H$. 
Denote by $\Sigma$ the set of primitive generators in $\mathcal M$ of the extremal rays of the negative of the dual of the valuation cone $\mathcal V$. 

\begin{remark} 
The normal equivariant embeddings of a given spherical homogeneous space are classified by combinatorial objects called \emph{colored fans}, 
which generalize the fans appearing in the classification of toric varieties. 
In a brief way, a colored fan is a finite collection of colored cones, 
which is a pair $(\mathcal C, \mathfrak R)$ consisting of $\mathfrak R \subset \mathfrak D$ and a strictly convex cone $\mathcal C \subset \mathcal N \otimes \mathbb Q$ generated by $\rho(\mathfrak R)$ and finitely many elements in the valuation cone $\mathcal V$ (see \cite{Knop91} for details).  
\end{remark} 

For a simple root $\alpha$, we denote by $P_{\alpha}$ the corresponding minimal parabolic subgroup containing $B$. 
We define the set $\mathfrak D (\alpha)$ of colors moved by the minimal parabolic subgroup $P_{\alpha}$: 
$$\mathfrak D (\alpha) = \{ D_i \in \mathfrak D : P_{\alpha} \cdot D_i \neq D_i \}.$$
Note that every color is moved by at least one minimal parabolic subgroup since the colors are not $G$-stable. 

Now, we recall the types of colors introduced by Brion \cite{Brion97} and Luna \cite{Luna97}. 

\begin{definition} 
\label{types of colors}
Suppose that a color $D_i$ is moved by a minimal parabolic subgroup $P_{\alpha}$, that is, $D_i \in \mathfrak D (\alpha)$. 
\begin{itemize}
	\item If $\alpha \in \Sigma$, we say that $D_i$ is of \emph{type \textsf{a}}.
	\item If $2 \alpha \in \Sigma$, we say that $D_i$ is of \emph{type \textsf{2a}}.
	\item Otherwise, we say that $D_i$ is of \emph{type \textsf{b}}.
\end{itemize}
\end{definition}

\begin{remark} 
The type of a color $D_i$ does not depend on the choice of a simple root $\alpha$ such that $D_i \in \mathfrak D (\alpha)$. 
\end{remark} 

From \cite[Theorem 4.2]{Brion97} and \cite[Section 3.6]{Luna97}, 
we can get an explicit expression of the anticanonical divisor $- K_X$ for a spherical variety $X$. 

\begin{proposition}
\label{expression of anticanonical divisor}
Let $X$ be an embedding of a spherical homogeneous space $G/H$. 
There exists a $B$-semi-invariant global section $s \in \Gamma(X, K_X^{-1})$ of the anticanonical bundle $K_X^{-1}$ such that 
$$\textup{div}(s) = \sum_{i=1}^k m_i D_i + \sum_{j=1}^{\ell} E_j$$ 
for colors $D_i$ and $G$-stable divisors $E_j$ in $X$. 
Denoting by $\xi \in \mathfrak X(B)$ the $B$-weight of the section $s$, 
the positive integers $m_i$ are obtained in terms of $\xi$ and the types of colors, more explicitly, 
 \[
  m_i = 
  \begin{cases}
  1 & \text{ if } D_i \text{ is of type \textsf{a} or \textsf{2a}}, \\
  \langle \alpha^{\vee}, \xi \rangle & \text{ if } D_i \text{ is of type \textsf{b} and } D_i \in \mathfrak D (\alpha). 
  \end{cases}
  \]
\end{proposition}

Let $P$ be the stabilizer of the open orbit of $B$ in a spherical $G$-variety $X$. 
Then $P$ is a parabolic subgroup of $G$ containing $B$, and admits a unique Levi subgroup $L$ containing $T$. 
We denote by $\Phi_{P^u}$ the set of roots of the unipotent radical $P^u$ of $P$, which are the roots of $P$ that are not roots of $L$. 
By \cite[Proposition~4.2]{GH15}, we have $\xi = \displaystyle \sum_{\alpha \in \Phi_{P^u}} \alpha = \sum_{\alpha \in \Phi^+ \backslash \Phi_{L}} \alpha$.
Alternatively, $\xi = 2 \rho_G - 2 \rho_L$, where $\rho_L$ is the sum of the fundamental weights of the root system of $L$. 

Based on the works of Brion \cite{Brion89, Brion97}, 
Gagliardi and Hofscheier \cite[Section 9]{GH15} described the (algebraic) moment polytope of the anticanonical line bundle on a Gorenstein Fano spherical variety. 

\begin{proposition}
\label{moment polytope of spherical variety}
Let $X$ be a Gorenstein Fano embedding of a spherical homogeneous space $G/H$. 
If a $B$-stable Weil divisor 
$-K_X = \sum_{i=1}^k m_i D_i + \sum_{j=1}^{\ell} E_j$ 
represents the anticanonical line bundle $K_X^{-1}$ for colors $D_i$ and $G$-stable divisors $E_j$ in $X$, 
the moment polytope $\Delta(X, K_X^{-1})$ is $\xi + Q_X^*$, 
where the polytope $Q_X$ is the convex hull of the set 
\begin{equation*}
\left\{ \frac{\rho(D_i)}{m_i} : i= 1, \cdots, k \right\} \cup \{ \hat{\rho}(E_j) 
: j= 1, \cdots , \ell \}
\end{equation*}  
in $\mathcal N \otimes \mathbb R$ and 
its dual polytope $Q_X^*$ is defined as 
$\{ m \in \mathcal M\otimes \mathbb R : \langle n, m \rangle \geq -1 \text{ for every } n \in Q_X \}$. 
\end{proposition}

\subsection{Criterion for existence of K\"{a}hler--Einstein metrics on symmetric varieties} 
\label{criterion for symmetric varieties}

For an algebraic group involution $\theta$ of a connected reductive algebraic group $G$, 
let $G^{\theta}=\{ g \in G : \theta(g)=g \}$ be the subgroup consisting of elements fixed by $\theta$. 
If $H$ is a closed subgroup of $G$ such that the identity component of $H$ coincides with the identity component of $G^{\theta}$, 
then the homogeneous space $G/H$ is called a \emph{symmetric homogeneous space}. 
By taking a universal cover of $G$, we can always assume that $G$ is simply connected. 
When $G$ is simply connected, 
$G^{\theta}$ is connected (see \cite[Section~8.1]{Steinberg68}) and 
$H$ is a closed subgroup between $G^{\theta}$ and its normalizer $N_G(G^{\theta})$ in $G$, that is, $G^{\theta} \subset H \subset N_G(G^{\theta})$. 

For an algebraic group involution $\theta$ of $G$, 
a torus $T$ in $G$ is \emph{split} if $\theta(t)=t^{-1}$ for any $t \in T$. 
A torus $T$ is \emph{maximally split} if $T$ is a $\theta$-stable maximal torus in $G$ which contains a split torus $T_s$ of maximal dimension among split tori. 
Then $\theta$ descends to an involution of $\mathfrak X(T)$ for a maximally split torus $T$, 
and the rank of a symmetric homogeneous space $G/H$ is equal to the dimension of a maximal split subtorus $T_s$ of $T$. 

Let $\Phi = \Phi(G, T)$ be the root system of $G$ with respect to a maximally split torus $T$. 
By \cite[Lemma~1.2]{dCP83}, 
we can take a set of positive roots $\Phi^+$ such that either $\theta(\alpha) = \alpha$ or $\theta(\alpha)$ is a negative root for all $\alpha \in \Phi^+$;  
then we denote $2 \rho_{\theta} = \sum_{\alpha \in \Phi^+ \backslash \Phi^{\theta}} \alpha$, where $\Phi^{\theta} = \{ \alpha \in \Phi : \theta(\alpha) = \alpha \}$. 
The set 
$$\Phi_{\theta} = \{ \alpha - \theta(\alpha) : \alpha \in \Phi \backslash \Phi^{\theta} \}$$ 
is a (possibly non-reduced) root system, 
which is called the \emph{restricted root system}.  
Let $\mathcal C^+_{\theta}$ denote the cone generated by positive restricted roots in $\Phi_{\theta}^+ = \{ \alpha - \theta(\alpha) : \alpha \in \Phi^+ \backslash \Phi^{\theta} \}$.

Vust proved that a symmetric homogeneous space $G/H$ is spherical as a consequence of the Iwasawa decomposition (see \cite[Theorem~1 in Section~1.3]{Vust74}). 
The spherical data of $G/H$ can be described in terms of the restricted root system $\Phi_{\theta}$ (see \cite{Vust90} and \cite[Section~26]{Timashev11}). 

\begin{proposition}
\label{spherical data of symmetric space}
For a symmetric homogeneous space $G/H$, 
\begin{itemize}
	\item the spherical weight lattice $\mathcal M$ is the lattice $\mathfrak X(T / T \cap H)$, 
	\item the valuation cone $\mathcal V$ is the negative restricted Weyl chamber determined by $-(\Phi_{\theta}^+)^{\vee}$ in $\mathcal N \otimes \mathbb Q$, and 
	\item the image $\rho(\mathfrak D(G/H))$ under the color map $\rho$ is the set of simple restricted coroots $\bar{\alpha}_i^{\vee}$. 
\end{itemize}
\end{proposition}

\begin{definition}
A normal $G$-variety $X$ together with an equivariant open embedding $G/H \hookrightarrow X$ 
of a symmetric homogeneous space $G/H$ is called a \emph{symmetric variety}. 
\end{definition}

Combining Propositions~\ref{expression of anticanonical divisor}, \ref{moment polytope of spherical variety}, \ref{spherical data of symmetric space}, 
we can easily compute (algebraic) moment polytopes of the anticanonical line bundles on Fano symmetric varieties. 

\begin{corollary}
\label{moment polytope}
Let $X$ be a Fano embedding of a symmetric homogeneous space $G/H$. 
If all colors in $G/H$ are of type \textsf{b} and 
$\langle \rho(D_i), 2 \rho_{\theta} \rangle = \langle \alpha^{\vee}, 2 \rho_{\theta} \rangle$ for any $D_i \in \mathfrak D (\alpha)$, 
then the moment polytope $\Delta(X_m, K_{X_m}^{-1})$ is the intersection of the positive restricted Weyl chamber and the half-spaces 
$\{ m \in \mathcal M \otimes \mathbb R : \langle \hat{\rho}(E_j), m - 2 \rho_{\theta} \rangle \geq -1 \}$ 
for all $G$-stable divisors $E_j$ in $X$. 
\end{corollary}

\begin{proof}
Suppose that a Weil divisor $-K_X = \sum_{i=1}^k m_i D_i + \sum_{j=1}^{\ell} E_j$ represents the anticanonical line bundle $K_X^{-1}$ for colors $D_i$ and $G$-stable divisors $E_j$ in $X$. 
If a color $D_i$ is of type \textsf{b} and $D_i \in \mathfrak D (\alpha)$ for a simple root $\alpha$, 
then we have $m_i = \langle \alpha^{\vee}, \xi \rangle$ by Proposition \ref{expression of anticanonical divisor}.
Note that when $X$ is a symmetric variety associated to an involution $\theta$ of $G$, 
the $B$-weight $\xi$ of a general global section in $\Gamma(X, K_X^{-1})^{(B)}$ is equal to $2 \rho_{\theta} = \sum_{\alpha \in \Phi^+ \backslash \Phi^{\theta}} \alpha$. 

As $\{ \rho(D_1), \cdots , \rho(D_k)\} = \{ \bar{\alpha}_1^{\vee}, \cdots , \bar{\alpha}_r^{\vee} \}$ from Proposition \ref{spherical data of symmetric space}, where $r$ denotes the rank of $\mathcal M$ and $k \geq r$, 
$\rho(D_i) = \bar{\alpha}_p^{\vee}$ gives an inequality 
$$\left\langle \frac{\bar{\alpha}_p^{\vee}}{\langle \alpha^{\vee}, 2 \rho_{\theta} \rangle}, m - 2 \rho_{\theta} \right\rangle \geq -1 
\quad \iff \quad \left\langle\bar{\alpha}_p^{\vee}, m \right\rangle \geq 0 \quad \text{ for all } 1 \leq p \leq r.$$   
Thus, the images of all colors $D_1, \cdots, D_k$ determine the positive restricted Weyl chamber, 
and the result follows from Proposition~\ref{moment polytope of spherical variety}. 
\end{proof}

Note that, for smooth Fano symmetric varieties, the linear part of the valuation cone $\mathcal{V}$ is trivial. 
Thus, Delcroix's criterion for $G$-equivariant K-stability of smooth Fano symmetric varieties in \cite[Corollary~5.9]{Del20} can be slightly strengthened as follows due to the observation in \cite[Proposition 5.10]{Golota20}. 

\begin{proposition}
\label{criterion}
Let $X$ be a smooth Fano embedding of a symmetric homogeneous space $G/H$ associated to an involution $\theta$ of $G$. 
Then the following are equivalent: 
\begin{enumerate}
    \item[\rm (1)] $X$ admits a K\"{a}hler--Einstein metric. 
    \item[\rm (2)] $X$ is $G$-equivariantly uniformly K-stable. 
    \item[\rm (3)] the barycenter of the moment polytope $\Delta(X, K_X^{-1})$ with respect to the Duistermaat--Heckman measure 
$$\prod_{\alpha \in \Phi^+ \backslash \Phi^{\theta}} \kappa(\alpha, p) \, dp$$ 
is in the relative interior of the translated cone $2 \rho_{\theta} + \mathcal C^+_{\theta}$, 
where $\kappa$ denotes the Cartan--Killing form on the Lie algebra $\mathfrak g$ of $G$.
\end{enumerate}
\end{proposition}

\subsection{Greatest Ricci lower bounds of smooth Fano symmetric varieties} 
For a Fano manifold $X$, the \emph{greatest Ricci lower bound} $R(X)$ of $X$ is defined as the supremum of all $0 \leq t \leq 1$ such that there exists a K\"{a}hler form $\omega$ in the first Chern class $c_1(X)$ with $\Ric(\omega) \geq t \, \omega$.
This invariant was first studied by Tian~\cite{tian92}, and was explicitely defined by Rubinstein~\cite{Rubinstein08, Rubinstein09}, where it was called Tian's $\beta$-invariant. 
It was further studied by Sz\'{e}kelyhidi~\cite{Sze11}; Song and Wang~\cite{SW16}. 
It is shown to be the same as the maximum existence time of Aubin and Yau's continuity path for finding a K\"{a}hler--Einstein metric. 
Given a K\"{a}hler form $\omega \in c_1(X)$, 
the continuity method involves introducing a family of equations 
$$\Ric(\omega_t) = t \, \omega_t + (1 - t) \omega \qquad \text{ for } t \in [0, 1]$$
depending on a parameter $t$ and finding a K\"{a}hler form $\omega_t$ to solve the equation, which for $t=1$ gives the Einstein equation we want to solve. 
Hence, in the absence of a K\"{a}hler--Einstein metric $R(X)$ may be regarded as a sort of numerical measure of how a Fano manifold $X$ fails to be K\"{a}hler--Einstein. 
Furthermore, the greatest Ricci lower bound $R(X)$ of a Fano manifold $X$ is closely related with the $\delta$-invariant defined by Fujita and Odaka \cite{FO18} using log canonical thresholds of basis type divisors. 
In fact, the basis log canonical threshold $\delta(X, -K_X)$ and $R(X)$ are related by the formula $R(X) = \min\{ 1, \delta(X, -K_X) \}$ by \cite[Corollary~7.6]{BBJ21} and \cite[Theorem 5.7]{CRZ19} (see also \cite[Section~3.3]{Golota20}). 

By definition, if $X$ admits a K\"{a}hler--Einstein metric, then the greatest Ricci lower bound $R(X)$ is equal to 1. 
However, the converse does not hold. 
For example, Tian \cite{tian97} constructed a small deformation of the Mukai--Umemura 3-fold whose general members satisfy $R(X)=1$ but are not K-polystable (see \cite[Section 3]{Sze11}). 

For any toric Fano manifold $X$, 
Li \cite{Li11} found an explicit formula for the greatest Ricci lower bound $R(X)$ purely in terms of the moment polytope associated to $X$.
In the case of smooth Fano equivariant compactifications of complex Lie groups, 
Delcroix \cite{Del17} obtained a formula for the greatest Ricci lower bound 
where the barycenter of the moment polytope with respect to the Lebesgue measure is replaced by the barycenter with respect to the Duistermaat--Heckman measure. 
Furthermore, the recent result of Delcroix and Hultgren \cite{DH21} extends the formula to the case of \emph{horosymmetric} manifolds introduced in \cite{Del20horosymmetric}, which is a class of spherical varieties including smooth horospherical varieties and smooth symmetric varieties. 

Now, we state an explicit expression of the greatest Ricci lower bound for smooth symmetric varieties using the following notion. 

\begin{definition} 
\label{toric polytope} 
Let $X$ be a smooth Fano embedding of a symmetric homogeneous space $G/H$. 
A \emph{toric polytope} $\Delta^{\text{tor}}$ of $X$ is defined as the convex hull of the images by the restricted Weyl group $\overline{W}$ of the moment polytope $\Delta(X, K_X^{-1})$.  
\end{definition}

The following proposition follows from \cite[Corollary 1.3]{DH21}. 

\begin{proposition} 
\label{Greatest Ricci lower bounds}
Let $X$ be a smooth Fano embedding of a symmetric homogeneous space $G/H$ associated to an involution $\theta$ of $G$. 
The greatest Ricci lower bound $R(X)$ of $X$ is equal to 
$$\sup \left \{ t \in (0, 1) : 2 \rho_{\theta} + \frac{t}{1-t} (2 \rho_{\theta} - \textbf{bar}_{DH}(\Delta)) \in \textup{Relint}(\Delta^{\text{tor}} - \mathcal C^+_{\theta}) \right \},$$ 
where $\textbf{bar}_{DH}(\Delta)$ is the barycenter of the moment polytope $\Delta(X, K_X^{-1})$ with respect to the Duistermaat--Heckman measure 
and $\textup{Relint}(\Delta^{\text{tor}} - \mathcal C^+_{\theta})$ means the relative interior of the Minkowski difference $\Delta^{\text{tor}} - \mathcal C^+_{\theta}$.  
\end{proposition}

We immediately get an elementary geometric expression for the greatest Ricci lower bound from the above result.

\begin{corollary} 
\label{formula for greatest Ricci lower bounds}
Suppose that a smooth Fano embedding $X$ of a symmetric homogeneous space associated to an involution $\theta$ of $G$ does not admit a K\"{a}hler--Einstein metric. 
Let $A$ be the point corresponding to $2 \rho_{\theta} $ in $\mathfrak X(T)$ and $C$ be $\textbf{bar}_{DH}(\Delta)$. 
If $Q$ is the point at which the half-line starting from the barycenter $C$ in the direction of $A$ intersects the boundary of $\Delta^{\text{tor}} - \mathcal C^+_{\theta}$, 
then the greatest Ricci lower bound of $X$ is equal to 
$$R(X) = \frac{\overline{AQ}}{\overline{CQ}}.$$  
\end{corollary}


\section{Wonderful compactifications of symmetric spaces of type AIII}

Let us recall the involution on $G=\SL_m(\mathbb C)$ of type AIII$(r, m)$ for $m \geq 2r$. 
Let $\theta$ be the involution of $G$ defined by $g \mapsto J_r g J_r$, 
where $J_r=\begin{pmatrix}
O & O & S_r \\
O & I_{m-2r} & O\\
S_r & O & O
\end{pmatrix}$ 
for the $r \times r$ matrix $(S_r)_{i, j}=\delta_{i+j, r+1}$. 
The subgroup $G^{\theta}$ fixed by $\theta$ is conjugate to the subgroup $S(\GL_r(\mathbb C) \times \GL_{m-r}(\mathbb C))$. 
If $m > 2r$, then the normalizer $N_G(G^{\theta})$ is equal to $G^{\theta}$. 
However, if $m = 2r$, then $N_G(G^{\theta})$ is different from $G^{\theta}$ and $N_G(G^{\theta}) / G^{\theta} \cong \mathbb Z_2$.

\subsection{Symmetric homogeneous spaces of type $\textup{AIII}(2, m)$} 

The homogeneous space $G/G^{\theta}$ is considered as the variety parametrizing pairs $(V_1, V_2)$ of linear subspaces in an $m$-dimensional complex vector space such that $\dim V_1 = r$, $\dim V_2 = m-r$, and $V_1 \cap V_2 = \{ \mathbf 0 \}$. 
Note that the dimension of the homogeneous space $\SL_m(\mathbb C)/S(\GL_r(\mathbb C) \times \GL_{m-r}(\mathbb C))$ is equal to $(m^2 - 1) - \{r^2 + (m-r)^2 -1\} = 2r(m-r)$. 
Since the Grassmannians $\Gr(r, m)$ and $\Gr(m-r, m)$ have a natural action of $\SL_m(\mathbb C)$ as homogeneous spaces, 
the symmetric space $G/G^{\theta} = \SL_m(\mathbb C)/S(\GL_r(\mathbb C) \times \GL_{m-r}(\mathbb C))$ is an open orbit for the diagonal action of $\SL_m(\mathbb C)$ on the product $\Gr(r, m) \times \Gr(m-r, m)$. 

For the root system $\mathsf{A}_{m-1}$ of $\SL_m(\mathbb C)$ with respect to the maximal torus $T$ of diagonal matrices, 
we know that the roots of $\mathsf{A}_{m-1}$ consist of algebraic group homomorphisms $\alpha_{i, j} \colon T \to \mathbb C^*$ defined by 
\[
\alpha_{i, j}(\diag(a_1, a_2, \cdots, a_{m-1}, a_{m})) = \frac{a_i}{a_j}
\] 
for $1 \leq i \neq j \leq m$. 
For simplicity, the simple roots $\alpha_{i, i+1}$ will be denoted by $\alpha_i$ for $1 \leq i \leq m-1$, 
and we use the same notations for the corresponding roots of the Lie algebra $\mathfrak{sl}_m(\mathbb C)$. 

\begin{figure}
 \begin{minipage}[b]{\textwidth}
 \centering

\begin{tikzpicture}
\clip (-2.3,-2.3) rectangle (3.3, 2.3); 

\coordinate (a1) at (1,-1);
\coordinate (a2) at (0,1);
\coordinate (a3) at ($(a1)+(a2)$);
\coordinate (a4) at ($2*(a3)$);
\coordinate (a5) at ($(a3)+(a2)$);
\coordinate (a6) at ($2*(a2)$);

\coordinate (v1) at (5,0);
\coordinate (v2) at (5,2);
\coordinate (v3) at (3.5, 3.5);

\coordinate (Origin) at (0,0);
\coordinate (asum) at ($(a1)+(a2)$);
\coordinate (2rho) at (4,2);

\foreach \x  in {-8,-7,...,12}{
  \draw[help lines,dashed]
    (\x,-8) -- (\x,11)
    (-8,\x) -- (11,\x) 
     [rotate=45] (1.414*\x/2,-8) -- (1.414*\x/2,12) ;
}

\foreach \x  in {-8,-7,...,12}{
  \draw[help lines,dashed]
     [rotate=135] (1.414*\x/2,-8) -- (1.414*\x/2,12) ;
}

\fill (Origin) circle (2pt) node[below left] {0};

\fill (a1) circle (2pt) node[below] {$\alpha_{1,m} - \alpha_{2,m-1}$};
\fill (a2) circle (2pt) node[left] {$\alpha_{2,m-1}$};
\fill (a3) circle (2pt) node[below] {$\alpha_{1,m}$};
\fill (a4) circle (2pt) node[below right] {$2\alpha_{1,m}$};
\fill (a5) circle (2pt) node[above right] {$\alpha_{1,m} + \alpha_{2,m-1}$};
\fill (a6) circle (2pt) node[left] {$2 \alpha_{2,m-1}$};

\draw[->,,thick](Origin)--(a1);
\draw[->,,thick](Origin)--(a2);
\draw[->,,thick](Origin)--(a3); 
\draw[->,,thick](Origin)--(a4);
\draw[->,,thick](Origin)--(a5);
\draw[->,,thick](Origin)--(a6); 

\end{tikzpicture} 

\caption{Positive restricted roots of type $\text{AIII}(2, m)$ for $m \geq 5$.}
\label{BC2}
\end{minipage}
\end{figure}
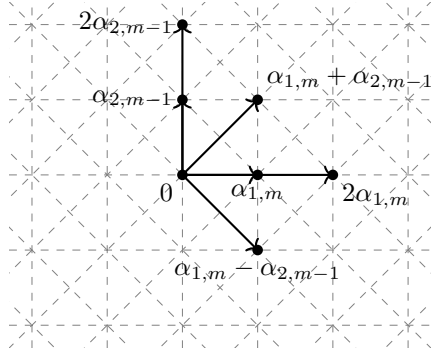

\begin{figure}
 \begin{minipage}[b]{\textwidth}
 \centering

\begin{tikzpicture}
\clip (-2.3,-2.3) rectangle (3.3, 2.3); 

\coordinate (a1) at (1,-1);
\coordinate (a2) at (0,1);
\coordinate (a3) at ($(a1)+(a2)$);
\coordinate (a4) at ($2*(a3)$);
\coordinate (a5) at ($(a3)+(a2)$);
\coordinate (a6) at ($2*(a2)$);

\coordinate (v1) at (5,0);
\coordinate (v2) at (5,2);
\coordinate (v3) at (3.5, 3.5);

\coordinate (Origin) at (0,0);
\coordinate (asum) at ($(a1)+(a2)$);
\coordinate (2rho) at (4,2);

\foreach \x  in {-8,-7,...,12}{
  \draw[help lines,dashed]
    (\x,-8) -- (\x,11)
    (-8,\x) -- (11,\x) 
     [rotate=45] (1.414*\x/2,-8) -- (1.414*\x/2,12) ;
}

\foreach \x  in {-8,-7,...,12}{
  \draw[help lines,dashed]
     [rotate=135] (1.414*\x/2,-8) -- (1.414*\x/2,12) ;
}

\fill (Origin) circle (2pt) node[below left] {0};

\fill (a1) circle (2pt) node[below] {$\alpha_{1,4} - \alpha_{2,3}$};
\fill (a4) circle (2pt) node[below right] {$2\alpha_{1,4}$};
\fill (a5) circle (2pt) node[above right] {$\alpha_{1,4} + \alpha_{2,3}$};
\fill (a6) circle (2pt) node[left] {$2 \alpha_{2,3}$};

\draw[->,,thick](Origin)--(a1);
\draw[->,,thick](Origin)--(a4);
\draw[->,,thick](Origin)--(a5);
\draw[->,,thick](Origin)--(a6); 

\end{tikzpicture} 

\caption{Positive restricted roots of type $\text{AIII}(2, 4)$.}
\label{C2}
\end{minipage}
\end{figure}
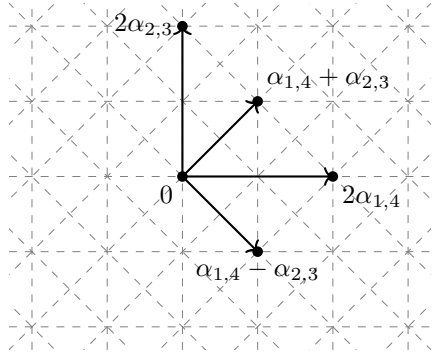

\begin{lemma} 
\label{restricted root system}
Let $\theta$ be the involution on $\SL_m(\mathbb C)$ of type $\textup{AIII}(2, m)$ for $m \geq 4$. 
The restricted root system $\Phi_{\theta}$ is as follow:
 \[
  \Phi_{\theta} = 
  \begin{cases}
  \{\pm \alpha_{1,m}, \pm \alpha_{2,m-1}, \pm 2 \alpha_{1,m}, \pm 2 \alpha_{2,m-1}, \pm(\alpha_{1,m} - \alpha_{2,m-1}), \pm (\alpha_{1,m} + \alpha_{2,m-1})\} \text{ of type $\mathsf{BC}_{2}$ }  & \text{ when } m \geq 5, \\
  \{\pm 2 \alpha_{1,4}, \pm 2 \alpha_{2,3}, \pm (\alpha_{1,4} - \alpha_{2,3}), \pm (\alpha_{1,4} + \alpha_{2,3})\} \text{ of type $\mathsf{C}_{2}$ } & \text{ when } m = 4.  
  \end{cases}
\]
Moreover, positive restricted roots $\alpha_{1,m}, \alpha_{2,m-1}, \alpha_{1,m} - \alpha_{2,m-1}, \alpha_{1,m} + \alpha_{2,m-1}$ have multiplicities $2(m-4)$, $2(m-4)$, $2$, $2$, respectively. 
\end{lemma}

\begin{proof}
Since the involution $\theta$ of type $\textup{AIII}(2, m)$ acts on diagonal matrices as 
\[
\theta(\diag(a_1, a_2, a_3, \cdots, a_{m-2}, a_{m-1}, a_{m})) = \diag(a_{m}, a_{m-1}, a_3, \cdots, a_{m-2}, a_{2}, a_{1}), 
\]
the roots $\alpha_{i, j}$ for $3 \leq i, j \leq m-2$ are fixed by the action of $\theta$. 
In particular, the cardinality of the set $\Phi_{\theta}^+ = \Phi^+ \backslash \Phi^{\theta}$ is equal to $\binom m 2 - \binom {m-4} 2 = 2(2m-5)$.

(i) $\alpha_{1,2} - \theta(\alpha_{1,2}) = \alpha_{1,2} - \alpha_{m,m-1} = \alpha_{1,m} - \alpha_{2,m-1} = \alpha_{m-1,m} - \alpha_{2,1}= \alpha_{m-1,m} - \theta(\alpha_{m-1,m})$ 
and 
$\alpha_{1,m-1} - \theta(\alpha_{1,m-1}) = \alpha_{1,m-1} - \alpha_{m,2} = \alpha_{1,m} + \alpha_{2,m-1} = \alpha_{2,m} - \alpha_{m-1,1}= \alpha_{2,m} - \theta(\alpha_{2,m})$.

(ii) For $3 \leq k \leq m-2$, 
 $\alpha_{1,k} - \theta(\alpha_{1,k}) = \alpha_{1,k} - \alpha_{m,k} = \alpha_{1,m} = \alpha_{k,m} - \alpha_{k,1} = \alpha_{k,m} - \theta(\alpha_{k,m})$ 
and 
 $\alpha_{2,k} - \theta(\alpha_{2,k}) = \alpha_{2,k} - \alpha_{m-1,k} = \alpha_{2,m-1} = \alpha_{k,m-1} - \alpha_{k,2} = \alpha_{k,m-1} - \theta(\alpha_{k,m-1})$. 

(iii) $\alpha_{1,m} - \theta(\alpha_{1,m}) = \alpha_{1,m} - \alpha_{m,1} = 2 \alpha_{1,m}$ and $\alpha_{2,m-1} - \theta(\alpha_{2,m-1}) = \alpha_{2,m-1} - \alpha_{m-1,2} = 2 \alpha_{2,m-1}$.
\end{proof}

\begin{lemma} 
\label{2rho}
For the involution $\theta$ on $\SL_m(\mathbb C)$ of type $\textup{AIII}(2, m)$, we have 
$2\rho_{\theta} = (m-1) \alpha_{1,m} + (m-3) \alpha_{2,m-1}$. 
\end{lemma}

\begin{proof}
From the proof of Lemma~\ref{restricted root system}, if $m \geq 5$ then we have 
\begin{align*}
2\rho_{\theta} & = \sum_{\alpha \in \Phi^+ \backslash \Phi^{\theta}} \alpha 
= \sum_{i=1}^{m-1}\sum_{j=1}^{i} \alpha_j + \sum_{i=2}^{m-1}\sum_{j=2}^{i} \alpha_j + \sum_{i=1}^{m-3}\sum_{j=1}^{i} \alpha_{m-j} + \sum_{i=2}^{m-3}\sum_{j=2}^{i} \alpha_{m-j} \\
& = (m-1) \alpha_{1,m} + (m-3) \alpha_{2,m-1}
\end{align*}
by the straightforward computations. 
When $m=4$, as $\Phi^{\theta} = \emptyset$ we can easily check that 
$$2\rho_{\theta} = \alpha_{1,4} + \alpha_{2,3} + (\alpha_{1,2} + \alpha_{2,4}) + (\alpha_{1,3} + \alpha_{3,4}) = 3 \alpha_{1,4} + \alpha_{2,3}.$$
Alternatively, it is possible to calculate $2\rho_{\theta}$ from the description of the restricted root system $\Phi_{\theta}$ in Lemma~\ref{restricted root system} 
because the sum of the positive restricted roots with multiplicities is equal to $2 \times 2\rho_{\theta}$. 
\end{proof}

From the detailed descriptions in \cite[Section 6.1]{Vust90}, if $m > 2r$ then we know that 
the symmetric homogeneous space $\SL_m(\mathbb C)/S(\GL_r(\mathbb C) \times \GL_{m-r}(\mathbb C))$ of type $\textup{AIII}(r, m)$ 
has $r+1$ colors $D_1, \cdots D_{r-1}, D^+, D^-$ such that 
\[
\mathcal D(\alpha_{1}) = \mathcal D(\alpha_{m-1}) = \{ D_1 \}, \cdots, 
\mathcal D(\alpha_{r-1}) = \mathcal D(\alpha_{m-r+1}) = \{ D_{r-1} \}, 
\mathcal D(\alpha_r) = \{ D^+ \} \text{ and } \mathcal D(\alpha_{m-r}) = \{ D^- \}. 
\]
On the other hand, if $m = 2r$, that is, $m-r = r$, then $\alpha_{m-r} = \alpha_r$ so that $D^+ = D^-$; 
hence the symmetric homogeneous space $\SL_{2r}(\mathbb C)/S(\GL_r(\mathbb C) \times \GL_r(\mathbb C))$ has $r$ colors $D_1, \cdots D_{r-1}, D^+ = D^-$. 

\begin{lemma} 
\label{type of colors}
For colors $D_1$, $D^+$, $D^-$ in the symmetric space $\SL_m(\mathbb C)/S(\GL_2(\mathbb C) \times \GL_{m-2}(\mathbb C))$ of type $\textup{AIII}(2, m)$, 
we have $\mathcal D(\alpha_{1}) = \mathcal D(\alpha_{m-1}) = \{ D_1 \}$, $\mathcal D(\alpha_2) = \{ D^+ \}$ and $\mathcal D(\alpha_{m-2}) = \{ D^- \}$. 
Moreover, all colors $D_1$, $D^+$, $D^-$ are of type \textsf{b}.
\end{lemma}

\begin{proof}
The first statement follows from \cite[Section 6.1]{Vust90}. 
As the valuation cone $\mathcal V$ of a symmetric homogeneous space is the negative restricted Weyl chamber determined by $-(\Phi_{\theta}^+)^{\vee}$ in $\mathcal N \otimes \mathbb Q$ 
by Proposition~\ref{spherical data of symmetric space}, 
$\mathcal V$ is the cone generated by two restricted coroots $\bar{\alpha}_{1,m}^{\vee}$ and $\bar{\alpha}_{1,m-1}^{\vee}$.  
Thus, the primitive generators in $\mathcal M$ of the extremal rays of $- \mathcal V^{\vee}$ are $\bar{\alpha}_{1,2}=\alpha_{1,m}-\alpha_{2,m-1}$ and $\bar{\alpha}_{2,3}=\alpha_{2,m-1}$, 
that is, $\Sigma = \{ \alpha_{1,m}-\alpha_{2,m-1}, \alpha_{2,m-1} \}$. 
For any simple root $\alpha$, neither $\alpha$ nor $2 \alpha$ belongs to $\Sigma$, which means that all colors are of type \textsf{b}. 
\end{proof}

\subsection{Wonderful compactifications of symmetric homogeneous spaces of type $\textup{AIII}(2, m)$} 

Wonderful varieties are $G$-varieties where the orbit structure is particularly nice. 
These varieties first appeared as compactifications with remarkable properties of symmetric homogeneous spaces in the work of De~Concini and Procesi~\cite{dCP83}, and were generalized by Luna~\cite{Luna96}. 

\begin{definition}
\label{Wonderful compactification}
An irreducible $G$-variety is called \emph{wonderful} of rank $r$ if it is smooth, complete, and it has exactly $r$ prime $G$-stable divisors such that they intersect transversally and their intersection is a (non-empty) single $G$-orbit. 
\end{definition}

By definition, a wonderful $G$-variety of rank $r$ has $2^r$ $G$-orbits; among them exactly one is open, and exactly one is closed. 
Note that any wonderful variety is spherical from \cite{Luna96}.  
In fact, wonderful varieties play a crucial role in the theory of spherical varieties; see \cite{Luna01, BP16, Pezzini18} for details. 

\begin{example}
(1) For a maximal torus $T \subset \SL_2(\mathbb C)$, the wonderful compactification of the symmetric space $\SL_2(\mathbb C)/T$ is exactly $\mathbb P^1 \times \mathbb P^1$. 

(2) The projective plane $\mathbb P^2 = \mathbb P(\Sym^2(\mathbb C^2))$ under the action of $\SL_2(\mathbb C)$ is a wonderful variety of rank 1, that is, the wonderful compactification of the symmetric space $\SL_2(\mathbb C)/N_{\SL_2}(T)$. 

(3) A classical example studied by De~Concini and Procesi~\cite{dCP83} is the variety of \emph{complete conics} $X = \{ ([A], [B]) \in \mathbb P(Mat_{3 \times 3}) \times \mathbb P(Mat_{3 \times 3}) : AB \in \mathbb C I_{3 \times 3} \}$. 
Then $X$ is the wonderful compactification of the homogeneous space $\SL_3(\mathbb C)/\SO_3(\mathbb C) Z(\SL_3(\mathbb C))$, which is the space of smooth conics in $\mathbb P^2$. 
Since $X$ has the two $\SL_3$-stable prime divisors consisting of non-invertible matrices, $X$ is wonderful of rank 2. 
\end{example}

For $m \geq 5$, the wonderful compactification of the symmetric space $\SL_m(\mathbb C)/S(\GL_2(\mathbb C) \times \GL_{m-2}(\mathbb C))$ is the blow-up of the product of Grassmannians $\Gr(2, m) \times \Gr(m-2, m)$ along the closed orbit $\mathcal{F \ell}(2, m-2; m)$. 
In the case of $m=4$, the symmetric homogeneous space $\SL_4(\mathbb C)/N_{\SL_4(\mathbb C)}(S(\GL_2(\mathbb C) \times \GL_2(\mathbb C)))$ admits the wonderful compactification. 
As the rank of the symmetric space is two, its wonderful compactification has the two $G$-stable divisors $Y_1, Y_2$. 

\begin{proposition} 
\label{anticanonical divisor}
Let $X_m$ be the wonderful compactification of the symmetric homogeneous space of type $\textup{AIII}(2, m)$ for $m \geq  4$. 
If $m \geq 5$, then there is a $B$-semi-invariant section of the anticanonical line bundle $K_{X_m}^{-1}$ with Weil divisor $-K_{X_m} = Y_1 + Y_2 + 2 D_1 + (m-3)D^+ + (m-3) D^-$, and the $B$-weight of this section is equal to $(m-1) \alpha_{1,m} + (m-3) \alpha_{2,m-1}$. 
Moreover, for the wonderful compactification $X_4$ of the symmetric homogeneous space of type $\textup{AIII}(2, 4)$, we have $-K_{X_4} = Y_1 + Y_2 + 2 D_1 + D^+$. 
\end{proposition}

\begin{proof}
First, let us consider the case $m \geq 5$. 
By \cite[Theorem 1.2]{GH15}, 
there is a $B$-semi-invariant section $s$ of the anticanonical line bundle $K_{X_m}^{-1}$ with Weil divisor $\text{div}(s) = Y_1 + Y_2 + m_1 D_1 + m^+ D^+ + m^- D^-$ and its $B$-weight is equal to $2\rho_{\theta} = (m-1) \alpha_{1,m} + (m-3) \alpha_{2,m-1}$ from Lemma~\ref{2rho}. 
As each color is of type \textsf{b} from Lemma~\ref{type of colors}, we obtain the coefficients
\begin{align*}
& m_1 = \langle \alpha_{1}^{\vee}, 2\rho_{\theta} \rangle = \langle \alpha_{m-1}^{\vee}, 2\rho_{\theta} \rangle = 2, \\
& m^+ = \langle \alpha_2^{\vee}, 2\rho_{\theta} \rangle = m-3, \\
& m^- = \langle \alpha_{m-2}^{\vee}, 2\rho_{\theta} \rangle = m-3 
\end{align*}
by \cite[Theorem 1.5]{GH15}. 
Thus, we get $-K_{X_m} = Y_1 + Y_2 + 2 D_1 + (m-3)D^+ + (m-3) D^-$. 

Similarly, as $-K_{X_4} = Y_1 + Y_2 + m_1 D_1 + m^+ D^+$, Lemmas~\ref{2rho} and \ref{type of colors} imply the result. 
\end{proof}

Two real roots $\alpha_{1,m}$ and $\alpha_{2,m-1}$ generate the spherical weight lattice $\mathcal M$.

\begin{proposition} 
\label{moment polytope_AIII}
Let $X_m$ be the wonderful compactification of the symmetric homogeneous space of type $\textup{AIII}(2, m)$ for $m \geq  4$. 
The moment polytope $\Delta_m = \Delta(X_m, K_{X_m}^{-1})$ of the anticanonical line bundle $K_{X_m}^{-1}$ is the convex hull of four points $0$, $m \alpha_{1,m}$, $m \alpha_{1,m} + (m-3) \alpha_{2,m-1}$, $(m - \frac{3}{2}) \alpha_{1,m} + (m - \frac{3}{2}) \alpha_{2,m-1}$ in $\mathcal M \otimes \mathbb R$. 
\end{proposition}

\begin{proof}
From the description in \cite[Section 6.1]{Vust90}, 
we know that the three colors $D_1, D^+, D^-$ and the $G$-stable divisors $Y_1, Y_2$ in $X_m$
have the images 
\begin{align*}
& \rho(D_1) = \bar{\alpha}_{1, 2}^{\vee} = \frac{1}{2} \alpha_{1,m}^{\vee}- \frac{1}{2} \alpha_{2,m-1}^{\vee}, \\
& \rho(D^+) = \rho(D^-) = \bar{\alpha}_{2, m-1}^{\vee} = \frac{1}{2} \alpha_{2,m-1}^{\vee}, \\
& \hat{\rho}(Y_1) =  - \bar{\alpha}_{1, m-1}^{\vee} = - \frac{1}{2} \alpha_{1,m}^{\vee}, \\ 
& \hat{\rho}(Y_2) =  - \bar{\alpha}_{1, m}^{\vee} = - \frac{1}{2} \alpha_{1,m}^{\vee} - \frac{1}{2} \alpha_{2,m-1}^{\vee}
\end{align*}
in $\mathcal N$, respectively.

Using Corollary \ref{moment polytope} and Proposition~\ref{anticanonical divisor}, 
$\frac{1}{2} \rho(D_1)$, $\frac{1}{m-3} \rho(D^+) = \frac{1}{m-3} \rho(D^-)$, $\hat{\rho}(Y_1)$ and $\hat{\rho}(Y_2)$ are used as inward-pointing facet normal vectors of the moment polytope $\Delta(X_m, K_{X_m}^{-1})$. 
First, $\frac{1}{m-3} \rho(D^+) = \frac{1}{m-3} \rho(D^-) = \frac{1}{2(m-3)} \alpha_{2,m-1}^{\vee} $ gives an inequality 
\[ 
\left\langle \frac{1}{2(m-3)} \alpha_{2,m-1}^{\vee}, x \alpha_{1,m} + y \alpha_{2,m-1} - 2 \rho_{\theta} \right\rangle = \frac{1}{m-3} \{ y - (m-3) \} \geq -1
\]  
because $2 \rho_{\theta} = (m-1) \alpha_{1,m} + (m-3) \alpha_{2,m-1}$. 
Similarly, as $\frac{1}{2} \rho(D_1) = \frac{1}{4} \alpha_{1,m}^{\vee}- \frac{1}{4} \alpha_{2,m-1}^{\vee}$ gives an inequality 
\[ 
\left\langle \frac{1}{4} \alpha_{1,m}^{\vee}- \frac{1}{4} \alpha_{2,m-1}^{\vee}, x \alpha_{1,m} + y \alpha_{2,m-1} - 2 \rho_{\theta} \right\rangle = \frac{1}{2} [ x-(m-1) - \{ y - (m-3) \}]  \geq -1, 
\]  
from which we get a domain $\{ x \alpha_{1,m} + y \alpha_{2,m-1} \in \mathcal M \otimes \mathbb R : x \geq y \}$. 
Thus, the images of two colors $D_1, D_2$ determine the positive restricted Weyl chamber. 
Next, for two $G$-stable divisors $\hat{\rho}(Y_1) = - \frac{1}{2} \alpha_{1,m}^{\vee}$ and $\hat{\rho}(Y_2) = - \frac{1}{2} \alpha_{1,m}^{\vee} - \frac{1}{2} \alpha_{2,m-1}^{\vee}$ give inequalities 
\[ 
\left\langle - \frac{1}{2} \alpha_{1,m}^{\vee}, x \alpha_{1,m} + y \alpha_{2,m-1} - 2 \rho_{\theta} \right\rangle = - \{ x-(m-1) \}  \geq -1 
\]  
and 
\[ 
\left\langle - \frac{1}{2} \alpha_{1,m}^{\vee} - \frac{1}{2} \alpha_{2,m-1}^{\vee}, x \alpha_{1,m} + y \alpha_{2,m-1} - 2 \rho_{\theta} \right\rangle = - \{ x-(m-1) \} - \{ y-(m-3) \}  \geq -1. 
\]  
Therefore, the moment polytope $\Delta(X_m, K_{X_m}^{-1})$ is the intersection of the positive restricted Weyl chamber and two half-spaces 
$\{ x \alpha_{1,m} + y \alpha_{2,m-1} \in \mathcal M \otimes \mathbb R : x \leq m \}$,  
$\{ x \alpha_{1,m} + y \alpha_{2,m-1} \in \mathcal M \otimes \mathbb R : x+y \leq 2m-3 \}$.
\end{proof}

For $p = x \alpha_{1,m} + y \alpha_{2,m-1} \in \mathcal M \otimes \mathbb R$, the Duistermaat--Heckman measure on $\mathcal M \otimes \mathbb R$ is given as   
\begin{align*}
\prod_{\alpha \in \Phi^+ \backslash \Phi^{\theta}} \kappa(\alpha, p) \, dp 
&= x^{2m-8} (2x) y^{2m-8} (2y) (x+y)^2 (x-y)^2 \, dxdy = 2^2 x^{2m-7} y^{2m-7} (x+y)^2 (x-y)^2 \, dxdy \\
&=: P_{DH}(x, y) dx dy 
\end{align*}
up to a multiplicative constant. 

\begin{figure}
 \begin{minipage}[b]{0.45 \textwidth}
 \centering

\begin{tikzpicture}
\clip (-1.3,-1.5) rectangle (7.3, 5.3); 

\coordinate (a1) at (1,-1);
\coordinate (a2) at (0,1);
\coordinate (a3) at ($(a1)+(a2)$);
\coordinate (a4) at ($2*(a3)$);
\coordinate (a5) at ($(a3)+(a2)$);
\coordinate (a6) at ($2*(a2)$);

\coordinate (v1) at (5,0);
\coordinate (v2) at (5,2);
\coordinate (v3) at (3.5, 3.5);

\coordinate (barycenter) at (4.37464513080842, 1.93589423571311);

\coordinate (Origin) at (0,0);
\coordinate (asum) at ($(a1)+(a2)$);
\coordinate (2rho) at (4,2);

\foreach \x  in {-8,-7,...,12}{
  \draw[help lines,dashed]
    (\x,-8) -- (\x,11)
    (-8,\x) -- (11,\x) 
     [rotate=45] (1.414*\x/2,-8) -- (1.414*\x/2,12) ;
}

\foreach \x  in {-8,-7,...,12}{
  \draw[help lines,dashed]
     [rotate=135] (1.414*\x/2,-8) -- (1.414*\x/2,12) ;
}

\fill (Origin) circle (2pt) node[below left] {0};

\fill (a1) circle (2pt) node[below] {$\alpha_{1,5} - \alpha_{2,4}$};
\fill (a2) circle (2pt) node[left] {$\alpha_{2,4}$};
\fill (a3) circle (2pt) node[below] {$\alpha_{1,5}$};
\fill (a4) circle (2pt) node[below right] {$2\alpha_{1,5}$};
\fill (a5) circle (2pt) node[below right] {$\alpha_{1,5} + \alpha_{2,4}$};
\fill (a6) circle (2pt) node[left] {$2 \alpha_{2,4}$};

\fill (2rho) circle (2pt) node[below left] {$2\rho_{\theta}$};

\fill (v1) circle (2pt) node[below] {$5 \alpha_{1, 5}$};
\fill (v2) circle (2pt) node[above right] {$5 \alpha_{1, 5} + 2 \alpha_{2, 4}$};
\fill (v3) circle (2pt) node[above left] {$\frac{7}{2} \alpha_{1, 5} + \frac{7}{2} \alpha_{2, 4}$};

\fill (barycenter) circle (2pt) node[below right] {$\textbf{bar}_{DH}(\Delta_5)$};

\draw[->,,thick](Origin)--(a1);
\draw[->,,thick](Origin)--(a2);
\draw[->,,thick](Origin)--(a3); 
\draw[->,,thick](Origin)--(a4);
\draw[->,,thick](Origin)--(a5);
\draw[->,,thick](Origin)--(a6); 

\draw[thick,gray](Origin)--(v1);
\draw[thick,gray](Origin)--(v3);
\draw[thick,gray](v1)--(v2);
\draw[thick,gray](v2)--(v3);

\draw [shorten >=-4cm, red, thick, dashed] (2rho) to ($(2rho)+(a1)$);
\draw [shorten >=-4cm, red, thick, dashed] (2rho) to ($(2rho)+(a2)$);
\end{tikzpicture} 

\caption{$\Delta_5=\Delta(X_5,K^{-1}_{X_5})$}
\label{Delta_5}
\end{minipage}

 \begin{minipage}[b]{.45 \textwidth}
 \centering

\begin{tikzpicture}
\clip (-1.3,-1.5) rectangle (8.3, 5.8); 

\coordinate (a1) at (1,-1);
\coordinate (a2) at (0,1);
\coordinate (a3) at ($(a1)+(a2)$);
\coordinate (a4) at ($2*(a3)$);
\coordinate (a5) at ($(a3)+(a2)$);
\coordinate (a6) at ($2*(a2)$);

\coordinate (v1) at (6,0);
\coordinate (v2) at (6,3);
\coordinate (v3) at (4.5, 4.5);

\coordinate (barycenter) at (5.40097490765179, 2.90084271414974);

\coordinate (Origin) at (0,0);
\coordinate (asum) at ($(a1)+(a2)$);
\coordinate (2rho) at (5,3);

\foreach \x  in {-8,-7,...,12}{
  \draw[help lines,dashed]
    (\x,-8) -- (\x,11)
    (-8,\x) -- (11,\x) 
     [rotate=45] (1.414*\x/2,-8) -- (1.414*\x/2,12) ;
}

\foreach \x  in {-8,-7,...,12}{
  \draw[help lines,dashed]
     [rotate=135] (1.414*\x/2,-8) -- (1.414*\x/2,12) ;
}

\fill (Origin) circle (2pt) node[below left] {0};

\fill (a1) circle (2pt) node[below] {$\alpha_{1,6} - \alpha_{2,5}$};
\fill (a2) circle (2pt) node[left] {$\alpha_{2,5}$};
\fill (a3) circle (2pt) node[below] {$\alpha_{1,6}$};
\fill (a4) circle (2pt) node[below right] {$2\alpha_{1,6}$};
\fill (a5) circle (2pt) node[below right] {$\alpha_{1,6} + \alpha_{2,5}$};
\fill (a6) circle (2pt) node[left] {$2 \alpha_{2,5}$};

\fill (2rho) circle (2pt) node[below left] {$2\rho_{\theta}$};

\fill (v1) circle (2pt) node[below] {$6 \alpha_{1, 6}$};
\fill (v2) circle (2pt) node[above right] {$6 \alpha_{1, 6} + 3 \alpha_{2, 5}$};
\fill (v3) circle (2pt) node[above left] {$\frac{9}{2} \alpha_{1, 6} + \frac{9}{2} \alpha_{2, 5}$};

\fill (barycenter) circle (2pt) node[below right] {$\textbf{bar}_{DH}(\Delta_6)$};

\draw[->,,thick](Origin)--(a1);
\draw[->,,thick](Origin)--(a2);
\draw[->,,thick](Origin)--(a3); 
\draw[->,,thick](Origin)--(a4);
\draw[->,,thick](Origin)--(a5);
\draw[->,,thick](Origin)--(a6); 

\draw[thick,gray](Origin)--(v1);
\draw[thick,gray](Origin)--(v3);
\draw[thick,gray](v1)--(v2);
\draw[thick,gray](v2)--(v3);

\draw [shorten >=-4cm, red, thick, dashed] (2rho) to ($(2rho)+(a1)$);
\draw [shorten >=-4cm, red, thick, dashed] (2rho) to ($(2rho)+(a2)$);
\end{tikzpicture} 

\caption{$\Delta_6=\Delta(X_6,K^{-1}_{X_6})$}
\label{Delta_6}
\end{minipage}
\end{figure}

\begin{example}
In the case of $m=5$, we compute the volume of the moment polytope $\Delta_5$ 
\begin{align*}
\text{Vol}_{DH}(\Delta_5) &= 
\displaystyle \int_{0}^{\frac{7}{2}} \int_{0}^{x} 4 x^{3} y^{3} (x+y)^2 (x-y)^2 \, dydx
+ \displaystyle \int_{\frac{7}{2}}^{5} \int_{0}^{7-x} 4 x^{3} y^{3} (x+y)^2 (x-y)^2 \, dydx
\\
& = \frac{1}{72}\left(\frac{7}{2}\right)^{12} + \left[ \frac{1}{72} x^{12} - \frac{5488}{27} x^9 + 4802 x^8 - 48020 x^7 + \frac{2235331}{9} x^6 - \frac{3294172}{5} x^5 + \frac{5764801}{8} x^4 \right]_{\frac{7}{2}}^{5} \\
& = \frac{391880669}{360}  
\end{align*}
and the barycenter of $\Delta_5$ with respect to the Duistermaat--Heckman measure 
\begin{align*}
\textbf{bar}_{DH}(\Delta_5)  
& = (\bar{x}, \bar{y}) 
= \frac{1}{\text{Vol}_{DH}(\Delta_5)} \left( \displaystyle \int_{\Delta_5} x \prod_{\alpha \in \Phi^+ \backslash \Phi^{\theta}} \kappa(\alpha, p) \, dp , \displaystyle \int_{\Delta_5} y \prod_{\alpha \in \Phi^+ \backslash \Phi^{\theta}} \kappa(\alpha, p) \, dp \right) \\
& = \left(\frac{1426329931935}{326044716608}, \frac{4418316612263}{2282313016256} \right) \approx (4.37, 1.94). 
\end{align*} 
Since $\bar{x}>4$ and $\bar{x} + \bar{y} > 6$, $\textbf{bar}_{DH}(\Delta_5)$ is in the relative interior of the translated cone $2 \rho_{\theta} + \mathcal C^+_{\theta}$ (see Figure~\ref{Delta_5}).  
Therefore, the wonderful compactification $X_5$ of the symmetric homogeneous space $\SL_5(\mathbb C)/S(\GL_2(\mathbb C) \times \GL_3(\mathbb C))$ admits a K\"{a}hler--Einstein metric by Proposition~\ref{criterion}.

Similarly, we obtain the barycenter of the moment polytope $\Delta_6$ with respect to the Duistermaat--Heckman measure 
\begin{align*}
\textbf{bar}_{DH}(\Delta_6)  
& = (\bar{x}, \bar{y}) = \left(\frac{421619272419}{78063549568}, \frac{226450079005}{78063549568} \right) \approx (5.40, 2.90). 
\end{align*} 
Since $\bar{x}>5$ and $\bar{x} + \bar{y} > 8$, $\textbf{bar}_{DH}(\Delta_6)$ is in the relative interior of the translated cone $2 \rho_{\theta} + \mathcal C^+_{\theta}$ (see Figure~\ref{Delta_6}).  
Therefore, the wonderful compactification $X_6$ of the symmetric homogeneous space $\SL_6(\mathbb C)/S(\GL_2(\mathbb C) \times \GL_4(\mathbb C))$ admits a K\"{a}hler--Einstein metric by Proposition~\ref{criterion}. 
\qed
\end{example}

\begin{proposition} 
\label{barycenter}
Let $(\bar{x}, \bar{y})$ be the barycenter $\textup{\textbf{bar}}_{DH}(\Delta_m)$ of the moment polytope $\Delta(X_m, K^{-1}_{X_m})$ with respect to the Duistermaat--Heckman measure. 
Then we have $\bar{x} > m-1$ and $\bar{x} + \bar{y} > 2m-4$ for all $m\geq 4$. 
\end{proposition}

\begin{proof}
From Proposition~\ref{moment polytope_AIII}, we can divide the moment polytope $\Delta_m$ into two regions $\Omega_1$ and $\Omega_2$: 
\begin{align*}
\Omega_1 & := \{ s (m, t) : 0 \leq s \leq 1, \, 0 \leq t \leq m-3 \}, \\
\Omega_2 & := \{ s (t, 2m-3-t) : 0 \leq s \leq 1, \, m - \frac{3}{2} \leq t \leq m \}. 
\end{align*}
Denoting $(\bar{x}_i, \bar{y}_i) = \textup{\textbf{bar}}_{DH}(\Omega_i)$ for each $i = 1, 2$, 
in order to prove $\bar{x} > m-1$ it suffices to show that $\bar{x}_i > m-1$ for each $i$. 
Indeed, $\bar{x}_1 > m-1$ and $\bar{x}_2 > m-1$ imply that $\bar{x} > m-1$ since 
\begin{align*}
\bar{x} & = \frac{1}{\text{Vol}_{DH}(\Delta_m)} \int_{\Delta_m} x \, P_{DH}(x, y) \, dx dy 
= \frac{1}{\text{Vol}_{DH}(\Omega_1 \cup \Omega_2)} \int_{\Omega_1 \cup \Omega_2} x \, P_{DH}(x, y) \, dx dy \\ 
& = \frac{\bar{x}_1 \text{Vol}_{DH}(\Omega_1) + \bar{x}_2 \text{Vol}_{DH}(\Omega_2)}{\text{Vol}_{DH}(\Omega_1) + \text{Vol}_{DH}(\Omega_2)}. 
\end{align*}

(i) Using the parametrization $x=ms, y=st$ of $\Omega_1$, we get the volume form 
\[
dx \wedge dy = d(ms) \wedge d(st) = m \, ds \wedge s \, dt = ms \, ds \wedge dt. 
\] 
Then we compute the volume of $\Omega_1$ with respect to the Duistermaat--Heckman measure: 
\begin{align*}
\text{Vol}_{DH}(\Omega_1) &= \int_{\Omega_1} P_{DH}(x, y) \, dx dy = \int_{\Omega_1} 2^2 x^{2m-7} y^{2m-7} (x+y)^2 (x-y)^2 \, dxdy \\
&= \int_0^1 \int_0^{m-3} 4 (ms)^{2m-7} (st)^{2m-7} \{s(m+t)\}^2 \{s(m-t)\}^2 \, ms \, dt ds \\
&= \int_0^1 \int_0^{m-3} 4 m^{2m-6} s^{4m-9} t^{2m-7} \{(m+t) (m-t)\}^2 \, dt ds \\
&= 4 m^{2m-6} \int_0^1 s^{4m-9} \, ds \int_0^{m-3} t^{2m-7} \{(m+t) (m-t)\}^2 \, dt \\
&= \frac{4 m^{2m-6}}{4m-8} \int_0^{m-3} t^{2m-7} (m+t)^2 (m-t)^2 \, dt. 
\end{align*}
Hence, when $m \geq 4$ we see that 
\begin{align*}
\bar{x}_1 &= \frac{1}{\text{Vol}_{DH}(\Omega_1)} \int_{\Omega_1} x \, P_{DH}(x, y) \, dx dy \\
&= \frac{1}{\text{Vol}_{DH}(\Omega_1)} \int_0^1 \int_0^{m-3} 4 m^{2m-5} s^{4m-8} t^{2m-7} \{(m+t) (m-t)\}^2 \, dt ds \\ 
&= \frac{m(4m-8)}{4m-7} > m-1. 
\end{align*}

(ii) Using the parametrization $x=st, y=s(2m-3-t)$ of $\Omega_2$, we get the volume form 
\begin{align*}
dx \wedge dy & = d(st) \wedge d(s(2m-3-t)) = (s \, dt + t \, ds) \wedge [(2m-3-t)ds + s(-dt)] \\
& = s(2m-3-t) \, dt \wedge ds -ts \, ds \wedge dt = s(2m-3) \, dt \wedge ds.  
\end{align*}
Then we compute the volume of $\Omega_2$ with respect to the Duistermaat--Heckman measure: 
\begin{align*}
\text{Vol}_{DH}(\Omega_2) &= \int_{\Omega_2} P_{DH}(x, y) \, dx dy = \int_{\Omega_2} 2^2 x^{2m-7} y^{2m-7} (x+y)^2 (x-y)^2 \, dxdy\\
&= \int_0^1 \int_{m-\frac{3}{2}}^{m} 4 (st)^{2m-7} \{s(2m-3-t)\}^{2m-7} \{s(2m-3)\}^2 \{s(2m-3-2t)\}^2 \, s(2m-3) \, dt ds \\
&= \int_0^1 \int_{m-\frac{3}{2}}^{m} 4 (2m-3)^3 s^{4m-9} t^{2m-7} (2m-3-t)^{2m-7} (2m-3-2t)^2 \, dt ds \\
&= 4 (2m-3)^3 \int_0^1 s^{4m-9} \, ds \int_{m-\frac{3}{2}}^{m} \{ t (2m-3-t)\}^{2m-7} (2m-3-2t)^2 \, dt \\
&= \frac{4 (2m-3)^3}{4m-8} \int_{m-\frac{3}{2}}^{m} \{ t (2m-3-t)\}^{2m-7} (2m-3-2t)^2 \, dt. 
\end{align*}
As $x+y = s(2m-3)$ on $\Omega_2$, we can verify the inequality
\begin{align*}
\bar{x}_2 + \bar{y}_2 &= \frac{1}{\text{Vol}_{DH}(\Omega_2)} \int_{\Omega_2} (x + y) \, P_{DH}(x, y) \, dx dy \\
&= \frac{1}{\text{Vol}_{DH}(\Omega_2)} \cdot 4 (2m-3)^4 \int_0^1 s^{4m-8} \, ds \int_{m-\frac{3}{2}}^{m} \{t (2m-3-t)\}^{2m-7} (2m-3-2t)^2 \, dt. \\ 
&= (2m-3) \cdot \frac{4m-8}{4m-7} = \frac{4m-6}{4m-7} \cdot (2m-4) > 2m-4.
\end{align*} 
Since 
\begin{align*}
\bar{x}_2 &= \frac{1}{\text{Vol}_{DH}(\Omega_2)} \int_{\Omega_2} x \, P_{DH}(x, y) \, dx dy \\
& = \frac{4m-8}{4m-7} \cdot \frac{\int_{m-\frac{3}{2}}^{m} t \{t (2m-3-t)\}^{2m-7} (2m-3-2t)^2 \, dt}{\int_{m-\frac{3}{2}}^{m} \{t (2m-3-t)\}^{2m-7} (2m-3-2t)^2 \, dt}, 
\end{align*} 
we conclude that $\bar{x}_2 >m-1$ by Lemma~\ref{inequality1}.  

(iii) To show $\bar{x} + \bar{y} > 2m-4$, for $m \geq 4$ we will verify the inequality 
\begin{align*}
\bar{x} + \bar{y} &= \frac{\displaystyle \int_{\Omega_1} (x+y) \, P_{DH}(x, y) \, dx dy + \int_{\Omega_2} (x+y) \, P_{DH}(x, y) \, dx dy}{\text{Vol}_{DH}(\Omega_1) + \text{Vol}_{DH}(\Omega_2)} \\
&= \frac{\displaystyle\frac{4 m^{2m-6}}{4m-7}  \int_0^{m-3} t^{2m-7} (m+t)^3 (m-t)^2 \, dt + \frac{4 (2m-3)^4}{4m-7} \int_{m-\frac{3}{2}}^{m} \{ t (2m-3-t)\}^{2m-7} (2m-3-2t)^2 \, dt}
{\displaystyle \frac{4 m^{2m-6}}{4m-8} \int_0^{m-3} t^{2m-7} (m+t)^2 (m-t)^2 \, dt + \frac{4 (2m-3)^3}{4m-8} \int_{m-\frac{3}{2}}^{m} \{ t (2m-3-t)\}^{2m-7} (2m-3-2t)^2 \, dt} \\
& > 2m-4.
\end{align*} 
To do this, it suffices to show the following inequality 
\[
m^{2m-6} \int_0^{m-3} t^{2m-7} (m+t)^2 (m-t)^2 (2t-2m+7) \, dt + (2m-3)^3 \int_{m-\frac{3}{2}}^{m} \{ t (2m-3-t)\}^{2m-7} (2m-3-2t)^2 \, dt >0. 
\]
From Lemma~\ref{inequality2}, we get the result as claimed. 
\end{proof}

Now, we complete the proof by showing the remaining two lemmas.

\begin{lemma} 
\label{inequality1}
For $m \geq 4$, the inequality 
\[ 
\int_{m-\frac{3}{2}}^{m} \left ( \frac{4m-8}{4m-7} \, t - (m-1) \right ) \{t (2m-3-t)\}^{2m-7} (2m-3-2t)^2 \, dt >0
\] 
holds. 
\end{lemma}

\begin{proof}
For a fixed $m$, since $t (2m-3-t)$ is a decreasing function on the interval $[m-\frac{3}{2},  m]$, we have $t (2m-3-t) > m(m-3)$ for $m-\frac{3}{2} < t <  m$. 
Thus, 
\begin{align*}
I(m) : & = \int_{m-\frac{3}{2}}^{m} \left ( \frac{4m-8}{4m-7} \, t - (m-1) \right ) \{t (2m-3-t)\}^{2m-7} (2m-3-2t)^2 \, dt \\
& > \int_{m-\frac{3}{2}}^{m} \left ( \frac{4m-8}{4m-7} \, t - (m-1) \right ) \{m (m-3)\}^{2m-7} (2m-3-2t)^2 \, dt  \\
& = \{m (m-3)\}^{2m-7} \int_{m-\frac{3}{2}}^{m} \left ( \frac{4m-8}{4m-7} \, t - (m-1) \right )  (2m-3-2t)^2 \, dt \\
& = \{m (m-3)\}^{2m-7} \cdot \frac{4m-8}{4m-7} \int_{0}^{\frac{3}{2}} \left ( u + m -\frac{3}{2} - (m-1) \frac{4m-7}{4m-8} \right )  (-2u)^2 \, du 
\qquad \text{(let $u:=t-(m-\frac{3}{2})$)}. 
\end{align*} 
Putting $\varphi(m) = (m-1) \displaystyle \frac{4m-7}{4m-8} - m + \frac{3}{2} 
= \frac{3}{4} + \frac{1}{4(m-2)}$, it suffices to show that $\displaystyle \int_{0}^{\frac{3}{2}} u^2(u-\varphi(m)) \, du >0$. 
As $\displaystyle \frac{3}{4} < \varphi(m) \leq \frac{7}{8}$ for $m \geq 4$, 
we obtain that $\displaystyle \int_{0}^{\frac{3}{2}} u^2(u-\varphi(m)) \, du = \left (\frac{3}{2} \right)^3 \left ( \frac{3}{8} - \frac{\varphi(m)}{3} \right) > 0$. 
\end{proof}

\begin{lemma} 
\label{inequality2}
For $m \geq 4$, the inequality 
\[ 
m^{2m-6} \int_0^{m-3} t^{2m-7} (m+t)^2 (m-t)^2 (2t-2m+7) \, dt + (2m-3)^3 \int_{m-\frac{3}{2}}^{m} \{ t (2m-3-t)\}^{2m-7} (2m-3-2t)^2 \, dt >0
\] 
holds. 
\end{lemma}

\begin{proof}
From the direct computation, we get 
\begin{align*}
& \int_0^{m-3} t^{2m-7} (m+t)^2 (m-t)^2 (2t-2m+7) \, dt \\
& = \int_0^{m-3} t^{2m-7} \{ 2t^5 - (2m-7) t^4 -4m^2 t^3 + 2(2m-7) m^2 t^2 + 2 m^4 t - (2m-7) m^4 \} \, dt \\
& = (m-3)^{2m-6} \cdot \frac{-512 m^7 + 9024 m^6 - 62848 m^5 + 233160 m^4 - 507384 m^3 + 656100 m^2 - 471420 m + 145800}{(2m-1)(2m-2)(2m-3)(2m-4)(2m-5)(2m-6)} \\
& =: (m-3)^{2m-6} q(m) 
\end{align*} 
As $t (2m-3-t) > m(m-3)$ for $m-\frac{3}{2} < t <  m$, we have 
\begin{align*}
& m^{2m-6} \int_0^{m-3} t^{2m-7} (m+t)^2 (m-t)^2 (2t-2m+7) \, dt + (2m-3)^3 \int_{m-\frac{3}{2}}^{m} \{ t (2m-3-t)\}^{2m-7} (2m-3-2t)^2 \, dt \\
& > m^{2m-6} \int_0^{m-3} t^{2m-7} (m+t)^2 (m-t)^2 (2t-2m+7) \, dt + (2m-3)^3 \int_{m-\frac{3}{2}}^{m} \{ m (m-3)\}^{2m-7} (2m-3-2t)^2 \, dt \\
& = m^{2m-6} (m-3)^{2m-6} q(m) + \frac{9}{2} \{ m (m-3) \}^{2m-7} (2m-3)^3 \\ 
& = m^{2m-7} (m-3)^{2m-7} \left\{ m (m-3) q(m) + \frac{9}{2} (2m-3)^3 \right\}.
\end{align*} 
We can check that a polynomial 
\begin{align*}
p(m) & := (2m-1)(2m-2)(2m-3)(2m-4)(2m-5)(2m-6) \left\{ m (m-3) q(m) + \frac{9}{2} (2m-3)^3 \right\} \\ 
& = m (m-3) (-512 m^7 + 9024 m^6 - 62848 m^5 + 233160 m^4 - 507384 m^3 + 656100 m^2 - 471420 m + 145800) \\
& \quad + \frac{9}{2} (2m-1)(2m-2)(2m-3)(2m-4)(2m-5)(2m-6) (2m-3)^3 
\end{align*} 
of degree 9 is increasing on $[4, \infty)$, and $p(4) > 0$. 
Hence, we conclude the result. 
\end{proof}

\begin{theorem}
\label{Main theorem-eq} 
All wonderful compactifications $X_m$ of symmetric homogeneous spaces of type $\textup{AIII}(2, m)$ for $m \geq  4$ are $\SL_m(\mathbb C)$-equivariantly uniformly K-stable.
\end{theorem}

\begin{proof}
Recall that $2 \rho_{\theta} = (m-1) \alpha_{1,m} + (m-3) \alpha_{2,m-1} = (m-1, m-3)$ from Lemma~\ref{2rho} and the cone $\mathcal C^+_{\theta}$ is generated by the vectors $(1, -1)$ and $(0, 1)$. 
As $\bar{x} > m-1$ and $\bar{x} + \bar{y} > 2m-4$ by Proposition~\ref{barycenter}, 
the barycenter $\textbf{bar}_{DH}(\Delta_m)$ is in the relative interior of the translated cone $2 \rho_{\theta} + \mathcal C^+_{\theta}$. 
Therefore, $X_m$ is $\SL_m(\mathbb C)$-equivariantly uniformly K-stable by Proposition \ref{criterion}. 
\end{proof}


\section{Blow-ups of the wonderful compactifications along the closed orbit} 

Let $Z = Y_1 \cap Y_2$ be the intersection of two $G$-stable divisors $Y_1, Y_2$ in the wonderful compactification $X_m$ of the symmetric homogeneous space of type $\textup{AIII}(2, m)$, which is the unique closed $G$-orbit in $X_m$. 
In this section, we study the existence of a K\"{a}hler--Einstein metric on the blow-up $B\ell_Z(X_m)$ of $X_m$ along the closed orbit $Z$. 

\subsection{Moment polytopes of the blow-ups of $X_m$ along the closed orbit} 
In the case of $m=4$, 
the blow-up $B\ell_Z(X_4)$ of the wonderful compactification $X_4$ of symmetric homogeneous space 
$$\SL_4(\mathbb C)/N_{\SL_4(\mathbb C)}(S(\GL_2(\mathbb C) \times \GL_2(\mathbb C)))$$ 
of type $\textup{AIII}(2, 4)$ along the unique closed orbit $Z$ is not Fano but Calabi--Yau from \cite[Proposition~2.1]{Ruzzi2012}.

\begin{proposition} 
\label{moment polytope_Bl}
For $m \geq 5$, the moment polytope $\Delta_m^{B\ell} = \Delta(B\ell_Z(X_m), K_{B\ell_Z(X_m)}^{-1})$ of the anticanonical line bundle $K_{B\ell_Z(X_m)}^{-1}$ is the convex hull of five points $0$, $m \alpha_{1,m}$, $m \alpha_{1,m} + (m-4) \alpha_{2,m-1}$, $(m-1) \alpha_{1,m} + (m-2) \alpha_{2,m-1}$, $(m - \frac{3}{2}) \alpha_{1,m} + (m - \frac{3}{2}) \alpha_{2,m-1}$ in $\mathcal M \otimes \mathbb R$. 
\end{proposition}

\begin{proof}
Since $Z$ is the intersection of two $G$-stable divisors $Y_1,Y_2$ in the wonderful compactification $X_m$, 
the exceptional divisor $E$ for the blow-up $\pi \colon B\ell_Z(X_m) \to X_m$ is linearly equivalent to $Y_1 + Y_2$.
As the $G$-stable divisors $Y_1, Y_2$ in $X_m$ have the images 
$\hat{\rho}(Y_1) = - \bar{\alpha}_{1, m-1}^{\vee} = - \frac{1}{2} \alpha_{1,m}^{\vee}$ and  
$\hat{\rho}(Y_2) = - \bar{\alpha}_{1, m}^{\vee} = - \frac{1}{2} \alpha_{1,m}^{\vee} - \frac{1}{2} \alpha_{2,m-1}^{\vee}$
in $\mathcal N$, respectively, 
the exceptional divisor $E$ for the blow-up $\pi \colon B\ell_Z(X_m) \to X_m$ has the image 
$\hat{\rho}(E) = - \bar{\alpha}_{1, m-1}^{\vee} - \bar{\alpha}_{1, m}^{\vee} \in \mathcal N$ (see \cite[Table~7]{Ruzzi2012}). 
The same argument as in Proposition \ref{moment polytope_AIII} implies the result. 
\end{proof}

\begin{figure}
 \begin{minipage}[b]{0.55 \textwidth}
 \centering

\begin{tikzpicture}
\clip (-1.3,-1.5) rectangle (7.3, 5.3); 

\coordinate (a1) at (1,-1);
\coordinate (a2) at (0,1);
\coordinate (a3) at ($(a1)+(a2)$);
\coordinate (a4) at ($2*(a3)$);
\coordinate (a5) at ($(a3)+(a2)$);
\coordinate (a6) at ($2*(a2)$);

\coordinate (v1) at (5,0);
\coordinate (v2) at (5,1);
\coordinate (v3) at (4,3);
\coordinate (v4) at (3.5, 3.5);

\coordinate (barycenter) at (4.13400232559311, 1.87005852772417);

\coordinate (Origin) at (0,0);
\coordinate (asum) at ($(a1)+(a2)$);
\coordinate (2rho) at (4,2);

\foreach \x  in {-8,-7,...,12}{
  \draw[help lines,dashed]
    (\x,-8) -- (\x,11)
    (-8,\x) -- (11,\x) 
     [rotate=45] (1.414*\x/2,-8) -- (1.414*\x/2,12) ;
}

\foreach \x  in {-8,-7,...,12}{
  \draw[help lines,dashed]
     [rotate=135] (1.414*\x/2,-8) -- (1.414*\x/2,12) ;
}

\fill (Origin) circle (2pt) node[below left] {0};

\fill (a1) circle (2pt) node[below] {$\alpha_{1,5} - \alpha_{2,4}$};
\fill (a2) circle (2pt) node[left] {$\alpha_{2,4}$};
\fill (a3) circle (2pt) node[below] {$\alpha_{1,5}$};
\fill (a4) circle (2pt) node[below right] {$2\alpha_{1,5}$};
\fill (a5) circle (2pt) node[below right] {$\alpha_{1,5} + \alpha_{2,4}$};
\fill (a6) circle (2pt) node[left] {$2 \alpha_{2,4}$};

\fill (2rho) circle (2pt) node[below left] {$2\rho_{\theta}$};

\fill (v1) circle (2pt) node[below] {$5 \alpha_{1, 5}$};
\fill (v2) circle (2pt) node[right] {$5 \alpha_{1, 5} + \alpha_{2, 4}$};
\fill (v3) circle (2pt) node[right] {$4 \alpha_{1, 5} + 3 \alpha_{2, 4}$};
\fill (v4) circle (2pt) node[above left] {$\frac{7}{2} \alpha_{1, 5} + \frac{7}{2} \alpha_{2, 4}$};

\fill (barycenter) circle (2pt) node[right] {$\textbf{bar}_{DH}(\Delta_5^{B\ell})$};

\draw[->,,thick](Origin)--(a1);
\draw[->,,thick](Origin)--(a2);
\draw[->,,thick](Origin)--(a3); 
\draw[->,,thick](Origin)--(a4);
\draw[->,,thick](Origin)--(a5);
\draw[->,,thick](Origin)--(a6); 

\draw[thick,gray](Origin)--(v1);
\draw[thick,gray](Origin)--(v4);
\draw[thick,gray](v1)--(v2);
\draw[thick,gray](v2)--(v3);
\draw[thick,gray](v3)--(v4);

\draw [shorten >=-4cm, red, thick, dashed] (2rho) to ($(2rho)+(a1)$);
\draw [shorten >=-4cm, red, thick, dashed] (2rho) to ($(2rho)+(a2)$);
\end{tikzpicture} 

\caption{$\Delta_5^{B\ell} = \Delta(B\ell_Z(X_5), K_{B\ell_Z(X_5)}^{-1})$}
\label{Delta_5_Bl}
\end{minipage}

 \begin{minipage}[b]{.55 \textwidth}
 \centering

\begin{tikzpicture}
\clip (-1.3,-1.5) rectangle (8.3, 5.8); 

\coordinate (a1) at (1,-1);
\coordinate (a2) at (0,1);
\coordinate (a3) at ($(a1)+(a2)$);
\coordinate (a4) at ($2*(a3)$);
\coordinate (a5) at ($(a3)+(a2)$);
\coordinate (a6) at ($2*(a2)$);

\coordinate (v1) at (6,0);
\coordinate (v2) at (6,2);
\coordinate (v3) at (5,4);
\coordinate (v4) at (4.5, 4.5);

\coordinate (barycenter) at (5.17130759356792, 2.80607873082343);

\coordinate (Origin) at (0,0);
\coordinate (asum) at ($(a1)+(a2)$);
\coordinate (2rho) at (5,3);

\foreach \x  in {-8,-7,...,12}{
  \draw[help lines,dashed]
    (\x,-8) -- (\x,11)
    (-8,\x) -- (11,\x) 
     [rotate=45] (1.414*\x/2,-8) -- (1.414*\x/2,12) ;
}

\foreach \x  in {-8,-7,...,12}{
  \draw[help lines,dashed]
     [rotate=135] (1.414*\x/2,-8) -- (1.414*\x/2,12) ;
}

\fill (Origin) circle (2pt) node[below left] {0};

\fill (a1) circle (2pt) node[below] {$\alpha_{1,6} - \alpha_{2,5}$};
\fill (a2) circle (2pt) node[left] {$\alpha_{2,5}$};
\fill (a3) circle (2pt) node[below] {$\alpha_{1,6}$};
\fill (a4) circle (2pt) node[below right] {$2\alpha_{1,6}$};
\fill (a5) circle (2pt) node[below right] {$\alpha_{1,6} + \alpha_{2,5}$};
\fill (a6) circle (2pt) node[left] {$2 \alpha_{2,5}$};

\fill (2rho) circle (2pt) node[below left] {$2\rho_{\theta}$};

\fill (v1) circle (2pt) node[below] {$6 \alpha_{1, 6}$};
\fill (v2) circle (2pt) node[right] {$6 \alpha_{1, 6} + 2 \alpha_{2, 5}$};
\fill (v3) circle (2pt) node[right] {$5 \alpha_{1, 6} + 4 \alpha_{2, 5}$};
\fill (v4) circle (2pt) node[above left] {$\frac{9}{2} \alpha_{1, 6} + \frac{9}{2} \alpha_{2, 5}$};

\fill (barycenter) circle (2pt) node[right] {$\textbf{bar}_{DH}(\Delta_6^{B\ell})$};

\draw[->,,thick](Origin)--(a1);
\draw[->,,thick](Origin)--(a2);
\draw[->,,thick](Origin)--(a3); 
\draw[->,,thick](Origin)--(a4);
\draw[->,,thick](Origin)--(a5);
\draw[->,,thick](Origin)--(a6); 

\draw[thick,gray](Origin)--(v1);
\draw[thick,gray](Origin)--(v4);
\draw[thick,gray](v1)--(v2);
\draw[thick,gray](v2)--(v3);
\draw[thick,gray](v3)--(v4);

\draw [shorten >=-4cm, red, thick, dashed] (2rho) to ($(2rho)+(a1)$);
\draw [shorten >=-4cm, red, thick, dashed] (2rho) to ($(2rho)+(a2)$);
\end{tikzpicture} 

\caption{$\Delta_6^{B\ell} = \Delta(B\ell_Z(X_6), K_{B\ell_Z(X_6)}^{-1})$}
\label{Delta_6_Bl}
\end{minipage}
\end{figure}

\begin{example}
\label{blowup-56}
In the case of $m=5$, we compute the volume of the moment polytope $\Delta_5^{B\ell}$ 
\begin{align*}
\text{Vol}_{DH}(\Delta_5^{B\ell}) &= 
\displaystyle \int_{0}^{\frac{7}{2}} \int_{0}^{x} 4 x^{3} y^{3} (x+y)^2 (x-y)^2 \, dydx
+ \displaystyle \int_{\frac{7}{2}}^{4} \int_{0}^{7-x} 4 x^{3} y^{3} (x+y)^2 (x-y)^2 \, dydx \\
&+ \displaystyle \int_{4}^{5} \int_{0}^{11-2x} 4 x^{3} y^{3} (x+y)^2 (x-y)^2 \, dydx
= \frac{1553111579}{2520}  
\end{align*}
and the barycenter of $\Delta_5^{B\ell}$ with respect to the Duistermaat--Heckman measure 
\begin{align*}
\textbf{bar}_{DH}(\Delta_5^{B\ell})  
& = (\bar{x}, \bar{y}) 
= \frac{1}{\text{Vol}_{DH}(\Delta_5^{B\ell})} \left( \displaystyle \int_{\Delta_5^{B\ell}} x \prod_{\alpha \in \Phi^+ \backslash \Phi^{\theta}} \kappa(\alpha, p) \, dp , \displaystyle \int_{\Delta_5^{B\ell}} y \prod_{\alpha \in \Phi^+ \backslash \Phi^{\theta}} \kappa(\alpha, p) \, dp \right) \\
& = \left(\frac{5341911643737}{1292188833728}, \frac{2416468747943}{1292188833728} \right) \approx (4.134, 1.870). 
\end{align*} 
Since $\bar{x}>4$ and $\bar{x} + \bar{y} > 6$, $\textbf{bar}_{DH}(\Delta_5^{B\ell})$ is in the relative interior of the translated cone $2 \rho_{\theta} + \mathcal C^+_{\theta}$ (see Figure~\ref{Delta_5_Bl}). 
Therefore, the blow-up $B\ell_Z(X_5)$ of $X_5$ along the closed orbit $Z$ admits a K\"{a}hler--Einstein metric by Proposition~\ref{criterion}.

Similarly, we obtain the barycenter of the moment polytope $\Delta_6^{B\ell}$ with respect to the Duistermaat--Heckman measure 
\begin{align*}
\textbf{bar}_{DH}(\Delta_6^{B\ell})  
& = (\bar{x}, \bar{y}) = \left(\frac{5817870364882097045}{1125028875118233728}, \frac{15784597990157403671}{5625144375591168640} \right) \approx (5.171, 2.806). 
\end{align*} 
Since $\bar{x} + \bar{y} < 8$, $\textbf{bar}_{DH}(\Delta_6^{B\ell})$ is not in the relative interior of the translated cone $2 \rho_{\theta} + \mathcal C^+_{\theta}$ (see Figure~\ref{Delta_6_Bl}). 
Therefore, the blow-up $B\ell_Z(X_6)$ of $X_6$ along the closed orbit $Z$ cannot admit any K\"{a}hler--Einstein metric by Proposition~\ref{criterion}. 
\qed 
\end{example}

\subsection{Barycenters of moment polytopes $\Delta_m^{B\ell}$} 
By Example~\ref{blowup-56}, it is enough to prove Theorem \ref{Blowup} when $m \geq 7$. 

\begin{proposition} 
\label{barycenter-blowup}
Let $(\bar{x}, \bar{y})$ be the barycenter $\textup{\textbf{bar}}_{DH}(\Delta_m^{B\ell})$ of the moment polytope $\Delta_m^{B\ell}$ with respect to the Duistermaat--Heckman measure. 
Then we have $\bar{x} + \bar{y} < 2m-4$ for $m \geq 6$. 
\end{proposition}

\begin{proof}
From Proposition \ref{moment polytope_Bl}, we can divide the moment polytope $\Delta_m^{B\ell}$ into three regions $\Omega_1$, $\Omega_2$ and $\Omega_3$: 
\begin{align*}
\Omega_1 & := \{ s (m, t) : 0 \leq s \leq 1, \, 0 \leq t \leq m-4 \}, \\
\Omega_2 & := \{ s (t, 3m-4-2t) : 0 \leq s \leq 1, \, m - 1 \leq t \leq m \}, \\
\Omega_3 & := \{ s (t, 2m-3-t) : 0 \leq s \leq 1, \, m - \frac{3}{2} \leq t \leq m-1 \}. 
\end{align*}
Using the parametrization $x=st, y=s(3m-4-2t)$ of $\Omega_2$, we get the volume form 
\begin{align*}
dx \wedge dy & = d(st) \wedge d(s(3m-4-2t)) = (s \, dt + t \, ds) \wedge [(3m-4-2t)ds + s(-2dt)] \\
& = s(3m-4-2t) \, dt \wedge ds -2ts \, ds \wedge dt = s(3m-4) \, dt \wedge ds.  
\end{align*}
Then we compute the volume of $\Omega_2$ with respect to the Duistermaat--Heckman measure: 
\begin{align*}
\text{Vol}_{DH}(\Omega_2) &= \int_{\Omega_2} P_{DH}(x, y) \, dx dy = \int_{\Omega_2} 2^2 x^{2m-7} y^{2m-7} (x+y)^2 (x-y)^2 \, dxdy\\
&= \int_0^1 \int_{m-1}^{m} 4 (st)^{2m-7} \{s(3m-4-2t)\}^{2m-7} \{s(3m-4-t)\}^2 \{s(3m-4-3t)\}^2 \, s(3m-4) \, dt ds \\
&= \int_0^1 \int_{m-1}^{m} 4 (3m-4) s^{4m-9} t^{2m-7} (3m-4-2t)^{2m-7} (3m-4-t)^2 (3m-4-3t)^2 \, dt ds \\
&= 4 (3m-4) \int_0^1 s^{4m-9} \, ds \int_{m-1}^{m} \{ t (3m-4-2t)\}^{2m-7} (3m-4-t)^2 (3m-4-3t)^2 \, dt \\
&= \frac{4 (3m-4)}{4m-8} \int_{m-1}^{m} \{ t (3m-4-2t)\}^{2m-7} (3m-4-t)^2 (3m-4-3t)^2 \, dt. 
\end{align*}
As $x+y = s(3m-4-t)$ on $\Omega_2$, we get
\begin{align*}
\bar{x}_2 + \bar{y}_2 &= \frac{1}{\text{Vol}_{DH}(\Omega_2)} \int_{\Omega_2} (x + y) \, P_{DH}(x, y) \, dx dy \\
&= \frac{1}{\text{Vol}_{DH}(\Omega_2)} \cdot 4 (3m-4) \int_0^1 s^{4m-8} \, ds \int_{m-1}^{m} \{ t (3m-4-2t)\}^{2m-7} (3m-4-t)^3 (3m-4-3t)^2 \, dt \\ 
&= \frac{1}{\text{Vol}_{DH}(\Omega_2)} \cdot \frac{4 (3m-4)}{4m-7} \int_{m-1}^{m} \{ t (3m-4-2t)\}^{2m-7} (3m-4-t)^3 (3m-4-3t)^2 \, dt.
\end{align*} 
Using the results in the proof of Proposition~\ref{barycenter}, we have 
\begin{align*}
\text{Vol}_{DH}(\Delta_m^{B\ell}) 
& = \displaystyle \frac{4 m^{2m-6}}{4m-8} \int_0^{m-4} t^{2m-7} (m+t)^2 (m-t)^2 \, dt \\ 
& + \frac{4 (3m-4)}{4m-8} \int_{m-1}^{m} \{ t (3m-4-2t)\}^{2m-7} (3m-4-t)^2 (3m-4-3t)^2 \, dt \\ 
& + \frac{4 (2m-3)^3}{4m-8} \int_{m-\frac{3}{2}}^{m-1} \{ t (2m-3-t)\}^{2m-7} (2m-3-2t)^2 \, dt, 
\end{align*} 
and 
\begin{align*}
(\bar{x}+\bar{y})\text{Vol}_{DH}(\Delta_m^{B\ell}) 
& = \displaystyle \frac{4 m^{2m-6}}{4m-7} \int_0^{m-4} t^{2m-7} (m+t)^3 (m-t)^2 \, dt \\ 
& + \frac{4 (3m-4)}{4m-7} \int_{m-1}^{m} \{ t (3m-4-2t)\}^{2m-7} (3m-4-t)^3 (3m-4-3t)^2 \, dt \\ 
& + \frac{4 (2m-3)^4}{4m-7} \int_{m-\frac{3}{2}}^{m-1} \{ t (2m-3-t)\}^{2m-7} (2m-3-2t)^2 \, dt. 
\end{align*} 
We can check $\bar{x} + \bar{y} < 2m-4$ for $6 \leq m \leq 41$ with the aid of a computer program \texttt{SAGE}. 
To prove $\bar{x} + \bar{y} < 2m-4$ for $m \geq 42$, it suffices to show the following inequality 
\begin{align*}
& m^{2m-6} \int_0^{m-4} t^{2m-7} (m+t)^2 (m-t)^2 (2t-2m+7) \, dt \\
& + (3m-4) \int_{m-1}^{m} \{ t (3m-4-2t)\}^{2m-7} (3m-4-t)^2 (3m-4-3t)^2 (2m-2t-1) \, dt \\
& + (2m-3)^3 \int_{m-\frac{3}{2}}^{m-1} \{ t (2m-3-t)\}^{2m-7} (2m-3-2t)^2 \, dt < 0. 
\end{align*} 

We first derive explicit expressions for each integrals. 
\begin{align*}
I_1 & = m^{2m-6} \int_0^{m-4} t^{2m-7} (m+t)^2 (m-t)^2 (2t-2m+7) \, dt \\
& = m^{2m-6} (m-4)^{2m-6} \\
& \quad \times \frac{-5888 m^7 + 82112 m^6 - 491008 m^5 + 1637752 m^4 - 3300288 m^3 + 4031616 m^2 - 2774016 m + 829440
}{(2m-1)(2m-2)(2m-3)(2m-4)(2m-5)(2m-6)} \\ 
& =: m^{2m-6} (m-4)^{2m-8} \cdot (m-4)^2 r(m). 
\end{align*} 
Since $2t-2m+7$ is negative for $0< t <m-4$, we know $I_1 < 0$ so that $r(m)<0$. 
By a well-known inequality $\left(1 - \displaystyle \frac{1}{n}\right)^n < e^{-1} < \left(1 - \displaystyle \frac{1}{n}\right)^{n-1}$ for all $n>1$, we obtain 
$$m^{2m-6} (m-4)^{2m-8} = m^{4m-14} \left(1 - \frac{4}{m}\right)^{2m-8} > m^{4m-14} (e^{-1})^{\frac{4}{m-4}(2m-8)} = m^{4m-14} e^{-8}.$$ 
Hence, $I_1 < m^{4m-14} e^{-8} (m-4)^2 r(m)$. 
For the second integral, we note that the term $2m-2t-1$ is negative for $m-\frac{1}{2} < t <m$, which implies that we have to separate two cases for the sake of estimate.  
\begin{align*}
I_2 & = (3m-4) \int_{m-1}^{m} \{ t (3m-4-2t)\}^{2m-7} (3m-4-t)^2 (3m-4-3t)^2 (2m-2t-1) \, dt \\
& < (3m-4) \Big[ \{ (m-1)(m-2) \}^{2m-7} \int_{m-1}^{m-\frac{1}{2}} (3m-4-t)^2 (3m-4-3t)^2 (2m-2t-1) \, dt \\ 
& \quad + \{ m(m-4) \}^{2m-7} \int_{m-\frac{1}{2}}^{m} (3m-4-t)^2 (3m-4-3t)^2 (2m-2t-1) \, dt \Big] \\
& = (3m-4) \Big [ \{ (m-1)(m-2) \}^{2m-7} \left (\frac{19}{8} m^2 - \frac{1837}{240}m + \frac{5929}{960} \right ) + \{ m(m-4) \}^{2m-7} \left (-\frac{99}{8} m^2 + \frac{11453}{240}m - \frac{44201}{960} \right ) \Big ]
\end{align*} 
Using inequalities 
\begin{align*}
\{ (m-1)(m-2) \}^{2m-7} 
& = m^{4m-14} \left(1 - \displaystyle \frac{1}{m}\right)^{2m-7} \left(1 - \displaystyle \frac{2}{m}\right)^{2m-7} \\
& < m^{4m-14} (e^{-1})^{2-\frac{7}{m}} (e^{-1})^{4-\frac{14}{m}} 
\leq m^{4m-14} e^{-6+\frac{1}{2}}
\end{align*} 
and 
$\{ m(m-4) \}^{2m-7} = m^{4m-14} \left(1 - \displaystyle \frac{4}{m}\right)^{2m-7} > m^{4m-14} (e^{-1})^{8+\frac{4}{m-4}} \geq m^{4m-14} e^{-8-\frac{2}{19}}$ for $m \geq 42$, 
we have 
$$I_2 < m^{4m-14} (3m-4)  \left \{ e^{-6+\frac{1}{2}} \left (\frac{19}{8} m^2 - \frac{1837}{240}m + \frac{5929}{960} \right ) + e^{-8-\frac{2}{19}} \left (-\frac{99}{8} m^2 + \frac{11453}{240}m - \frac{44201}{960} \right ) \right \}.$$ 
For the third term, we use the same method:  
\begin{align*}
I_3 & = (2m-3)^3 \int_{m-\frac{3}{2}}^{m-1} \{ t (2m-3-t)\}^{2m-7} (2m-3-2t)^2 \, dt \\
& < (2m-3)^3 \left(m - \frac{3}{2}\right)^{4m-14} \int_{m-\frac{3}{2}}^{m-1} (2m-3-2t)^2 \, dt = \frac{1}{6} (2m-3)^3 \left(m - \frac{3}{2}\right)^{4m-14} \\
& < \frac{1}{6} (2m-3)^3 \cdot m^{4m-14} e^{\frac{21}{m}-6} \leq \frac{1}{6} (2m-3)^3 m^{4m-14} e^{-6+\frac{1}{2}} 
\end{align*} 
for $m \geq 42$. 
Combining the calculations, we get the inequality 
\begin{align*}
I_1 + I_2 + I_3 & < m^{4m-14} e^{-8} \Big[ (m-4)^2 r(m) \\
& + (3m-4)  \left \{ e^{\frac{5}{2}} \left (\frac{19}{8} m^2 - \frac{1837}{240}m + \frac{5929}{960} \right ) + e^{-\frac{2}{19}} \left (-\frac{99}{8} m^2 + \frac{11453}{240}m - \frac{44201}{960} \right ) \right \} + \frac{e^{\frac{5}{2}}}{6} (2m-3)^3 \Big ]. 
\end{align*} 
Consequently, it suffices to show that a polynomial 
\begin{align*}
R(m) & := (m-4)^2 (-5888 m^7 + 82112 m^6 - 491008 m^5 + 1637752 m^4 - 3300288 m^3 + 4031616 m^2 - 2774016 m + 829440) \\
& + (2m-1)(2m-2)(2m-3)(2m-4)(2m-5)(2m-6) \\
& \times \Big [ (3m-4)  \left \{ e^{\frac{5}{2}} \left (\frac{19}{8} m^2 - \frac{1837}{240}m + \frac{5929}{960} \right ) + e^{-\frac{2}{19}} \left (-\frac{99}{8} m^2 + \frac{11453}{240}m - \frac{44201}{960} \right ) \right \} + \frac{e^\frac{5}{2}}{6} (2m-3)^3 \Big ]
\end{align*}  
of degree 9 is always negative for $m \geq 42$. 
We can check that the polynomial $R(m)$ is decreasing on $[42, \infty)$, and $R(42) < 0$. 
Therefore, we conclude the result. 
\end{proof}

\begin{theorem}
\label{Blowup-eq}
Let $X_m$ be the wonderful compactification of symmetric homogeneous space of type $\textup{AIII}(2, m)$ for $m \geq 5$. 
The blow-up $B\ell_Z(X_m)$ of $X_m$ along the closed orbit $Z$ is $\SL_m(\mathbb C)$-equivariantly uniformly K-stable if $m=5$ and K-unstable if $m \geq 6$.  
\end{theorem}

\begin{proof}
For $\textbf{bar}_{DH}(\Delta_m^{B\ell}) = (\bar{x}, \bar{y})$, if $m\geq 6$ then $\textbf{bar}_{DH}(\Delta_m^{B\ell})$ is not in the relative interior of the translated cone $2 \rho_{\theta} + \mathcal C^+_{\theta}$ as $\bar{x} + \bar{y} < 2m-4$ by Proposition~\ref{barycenter-blowup}. 
Thus, we get the result from Proposition~\ref{criterion} and Example~\ref{blowup-56}.  
\end{proof}

\subsection{Greatest Ricci lower bounds of the blow-ups $B\ell_Z(X_m)$} 
From Corollary~\ref{formula for greatest Ricci lower bounds}, we can compute the greatest Ricci lower bound of the blow-up of $X_m$ along the closed orbit for $m \geq 6$. 

\begin{example}
\label{GRLB_6}
We already know that the blow-up $B\ell_Z(X_6)$ of $X_6$ along the closed orbit $Z$ cannot admit any K\"{a}hler--Einstein metric. 
Let $Q$ be the point at which the half-line starting from the barycenter $C=\textbf{bar}_{DH}(\Delta_6^{B\ell})$ in the direction of $A=(5, 3)$ intersects the boundary of $(\Delta_6^{B\ell})^{\text{tor}} - \mathcal C^+_{\theta}$. 
Considering the line $x+y=9$ giving a part of $\partial((\Delta_6^{B\ell})^{\text{tor}} - \mathcal C^+_{\theta})$ and the half-line $\overrightarrow{CA}$, 
from the results in Example~\ref{blowup-56} we can compute 
$$Q = \left(-\frac{327603995647340905}{127205190161460224}, \frac{1472450707100482921}{127205190161460224} \right) \approx (-2.575, 11.575).$$ 
Thus, we obtain the greatest Ricci lower bound $R(B\ell_Z(X_6)) = \displaystyle \frac{\overline{AQ}}{\overline{CQ}} \approx 0.978$ 
by Corollary~\ref{formula for greatest Ricci lower bounds}.
\end{example}

\vskip 3em


\providecommand{\bysame}{\leavevmode\hbox to3em{\hrulefill}\thinspace}
\providecommand{\MR}{\relax\ifhmode\unskip\space\fi MR }
\providecommand{\MRhref}[2]{%
  \href{http://www.ams.org/mathscinet-getitem?mr=#1}{#2}
}
\providecommand{\href}[2]{#2}

\end{document}